\documentclass[11 pt]{amsart}
\usepackage[all]{xy}

\usepackage{amssymb,amsmath,color}
\usepackage{epsfig}
\theoremstyle{theorem}
\newtheorem{theorem}{Theorem}[section]
\newtheorem{prop}[theorem]{Proposition}

\newtheorem{co}[theorem]{Corollary}
\newtheorem{lemma}[theorem]{Lemma}
\newtheorem{cor}[theorem]{Corollary}

\theoremstyle{definition}
\newtheorem{defn}{Definition}

\theoremstyle{remark}
\newtheorem{remark}{Remark}
\newtheorem{eg}{Example}

\numberwithin{equation}{section}

\catcode`,\active

\catcode`\,12

\catcode`,\active

\catcode`\,12

\catcode`,\active

\catcode`\,12

\def\t{\textnormal}

\def\g{{\gamma}}

\def\O{{\omega}}

\def\[{\left[}
\def\]{\right]}
\def\[{\left[}
\def\]{\right]}
\def\({\left(}
\def\){\right)}

\def \fp{\frak p}

\def\C{\mathbb{C}}
\def\R{\mathbb{R}}
\def\Z{\mathbb{Z}}
\def\Q{\mathbb{Q}}
\def\F{\mathbb{F}}
\def\O{\mathcal{O}}

\def\GL2{\text{GL_2}}
\def\GL1{\text{GL_1}}

\newcommand{\BZ}{{\mathbb Z}}

\newcommand{\G}{{\Gamma}}
\newcommand{\SL}{{\text{SL}_2(\BZ)}}

  \let\g\gamma
\let\e\varepsilon    \let\k\chi

\def\C{\mathbb C}
\def\G{\mathbf G}

\def\bk{\mathbf k}

\def\CL{\mathcal L}
\def\E{\mathcal E}
\def\T{\mathcal T}

\def\CL{\mathcal L}

\def \Ind {\textnormal{Ind}}

\def\g{\gamma}
\def \gal{\textnormal{Gal}}

\def \GL2 {{\text{GL}_2}}

\def \inj{\hookrightarrow}

\def \gp{\mathfrak p}

\def \q {{\mathfrak q}}
\def\Tr{{\rm Tr}}
\def\Gal{{\rm Gal}}
\def\AF{{\rm Fr}}
\def\Frob{{\rm Frob}}
\def\F{{\mathbb F}}
  
\def\Z{{\mathbb Z}}

\def\Q{{\mathbb Q}}

\def\O{\mathcal O}

\def \R {\mathcal R}
\def\G{\Gamma}

\def\Qlbar {\overline{\Q}_\ell}
\def \GL {\textnormal{GL}}

\def \sym {\textnormal{Sym}}
\def \HT {{\rm HT}}

\def \tv {{\rm{v}}}

\def \wt {\widetilde}

\def \bk{ \color{black} }

\def \k {\kappa}

\def \bk {\color{black}}

\title[Galois representations  associated to noncongruence modular forms]{Potentially $\GL_2$-type Galois representations  associated to noncongruence modular forms }

\author{Wen-Ching Winnie Li}
\address{Department of Mathematics, Pennsylvania State University,
University Park, PA 16802, USA} \email{wli@math.psu.edu}

\author{Tong Liu}
\address{Department of Mathematics, Purdue University, West Lafayette, IN 47907, USA}\email{tongliu@math.purdue.edu}

\author{Ling Long}
\address{Department of Mathematics, Louisiana State University, Baton Rouge, LA 70803, USA}
\email{llong@lsu.edu}

\date{}
\subjclass[2000]{ 11F11,11F80}
\keywords{Galois representations, noncongruence modular forms, automorphy}
\thanks{Li was supported in part by NSF grant DMS \#1101368, Simons Foundation grant \# 355798, and MOST grant 105-2811-M-001-002, Liu by NSF grant DMS \#1406926, and Long by NSF grants DMS \#1303292 and \#1602047. Part of the paper was written when Li and Long were research fellows at the Institute for Computational and Experimental Research in Mathematics (ICERM) in Fall 2015. They would like to thank ICERM for its hospitality. Li would also like to thank the Institute of Mathematics, Academia Sinica in Taiwan for its support and hospitality in spring 2016 when she worked on this paper.
}

\begin{document}

\begin{abstract}
In this paper, we consider  representations of  the absolute Galois group $\text{Gal}(\overline {\mathbb Q}/\mathbb Q)$ attached to modular forms for noncongruence subgroups of $\text{SL}_2(\mathbb Z)$.  When the underlying modular curves have a model over $\Q$, these representations are constructed by Scholl in \cite{sch85b} and are referred to as Scholl representations, which form a large class of motivic Galois representations.  In particular, by a result of Belyi, Scholl representations include the Galois actions on the Jacobian varieties of algebraic curves defined over $\mathbb Q$.  As Scholl representations are motivic, they are expected to correspond to automorphic representations according to the Langlands philosophy. Using recent developments on automorphy lifting theorem, we obtain various automorphy and potential automorphy results for potentially $\GL_2$-type Galois representations  associated to noncongruence modular forms. Our results are applied to various kinds of examples. Especially, we obtain potential automorphy results for  Galois representations attached to an infinite family of spaces of weight 3 noncongruence cusp forms of  arbitrarily large dimensions.
\end{abstract}

\maketitle
\section{Introduction}In a series of papers \cite{LLY,ALL,Long08,HLV}, the authors investigated 4-dimensional $\ell$-adic representations of $G_\Q:=\text{Gal}(\overline \Q/\Q)$ arising from noncongruence cusp forms constructed by Scholl \cite{sch85b} which admit quaternion multiplication (QM) as given in Definition \ref{def:QM}. In each case, it is shown that Galois representations are automorphic in the sense that they correspond to automorphic representations of $\GL_4$ over $\Q$ as described by the Langlands program. (See Definition \ref{def:automorphic} for more details on automorphy and potential automorphy.) These automorphy results are obtained via the Faltings-Serre method which boils down to comparing the  Galois representations associated with noncongruence modular forms and those attached to automorphic targets  which were identified by extensive search in the modular form database. In \cite{ALLL}, a general automorphy result for 4-dimensional Scholl representations admitting quaternion multiplication was obtained using then newly established Serre conjecture and modularity lifting theorems. In the literature, there are many constructions of 4-dimensional Scholl representations with QM including families of 2-dimensional abelian varieties with QM (in the sense that their endomorphism algebra contains a quaternion algebra) which are parameterized by Shimura curves \cite{Shimura, Deligne}.   In \cite{DFLST}, the authors used hypergeometric functions over finite fields to study some natural families of Galois representations which are potentially of $\GL_2$-type (see Defintion \ref{def:GL2} below). In particular, they constructed
 a family of 8-dimensional  $\ell$-adic Galois representations $\{\rho_{\ell,t}\}$ of $G_\Q$ such that for each parameter $t\in \Q\smallsetminus \{0,1\}$ there is a Galois extension $F(t)/\Q$ depending on $t$ and upon enlarging the scalar field the restrictions decompose as $$\rho_{\ell,t}|_{G_{F(t)}}\cong \tilde{\sigma}_{\ell,t}\oplus\tilde{\sigma}_{\ell,t}\oplus\tilde{\sigma}_{\ell,t}\oplus\tilde{\sigma}_{\ell,t}.$$ Here and later $G_F$ denotes the absolute Galois group of the field $F$, and $\tilde{\sigma}_{\ell,t}$ is an irreducible 2-dimensional $\ell$-adic representation of $G_{F(t)}$. More details are given in \S \ref{ss:wt-2}. Representations behaving like $\rho_{\ell,t}$ are called potentially 2-isotypic (cf. Definition \ref{def:GL2}). For example, 4-dimensional Scholl representations with QM are potentially 2-isotypic. Such representations can be described quite explicitly using Clifford theory recalled in \S \ref{sec:clifford}.

 The aim of this paper is to establish the potential automorphy of Scholl representations potentially of $\GL_2$-type or
potentially 2-isotypic using recent important advances in automorphy lifting theorem. It should be pointed out that  most of the known (potential) automorphy criteria are for regular representations, namely those with distinct Hodge-Tate weights. On the other hand, the Scholl representations attached to a $d$-dimensional space of cusp forms of weight $\kappa \ge 2$ for a finite index subgroup of $\SL$ have Hodge-Tate weights $0$ and $1-\kappa$, each with multiplicity $d$ (cf. \cite{sch85b}). Hence they are highly irregular when $d>1$.  This paper is motivated by the hope that, for a Scholl representation potentially of $\GL_2$-type, there is a good chance that the Hodge-Tate weights would be evenly distributed among the $2$-dimensional subrepresentations so that the known criteria could be applied to conclude its potential automorphy.

Our main results are as follows.
\bigskip

\noindent {\bf Theorem A.} (Theorem \ref{thm-try})  {\it Let $F$ be a totally real field and $\{ \eta_\ell : G_F \to \GL _2 (\Qlbar) \}$ be a system of 2-dimensional $\ell$-adic Galois representations for all primes $\ell$. Suppose there exist a finite extension $K$ of $F$ and a compatible system of Scholl representations $\{\rho_\ell\}$ such that for each $\ell$,  $ \eta_\ell | _{G_K}$ is a  subrepresentation of $\rho_\ell|_{G_K}$.  Then there exists a set $\CL$ of {rational} primes of Dirichlet density 1 so that for each $\ell \in \CL$, $\eta_\ell$ is \emph{totally odd},
 i.e. $\det \eta_\ell (c) = -1$ for any complex conjugation $c \in G_F$,
and potentially automorphic. Moreover, $\eta_\ell$ is automorphic if $F=\Q$ or $\eta_\ell$ is potentially reducible.}
\bigskip

Here a representation of a group $G$ is called {\it potentially reducible} if it is reducible when restricted to a finite index subgroup of $G$.

As consequences of Theorem A, one obtains sufficient conditions for (potential) automorphy of Scholl representations, Proposition \ref{prop-save} and Corollary \ref{co-save}, which are applied to examples studied in this paper.
 Theorem A is also used to prove
\bigskip

\noindent {\bf Theorem B.} (Theorem \ref{thm:main2})  {\it Let $\{\rho_\ell\}$ be a compatible system of $2d$-dimensional semi-simple subrepresentations of Scholl representations of $G_\Q$. Assume that there is a finite Galois extension $F/\Q$ such that all $\rho_\ell$ are 2-isotypic when restricted to $G_F$.
Suppose that $F$ contains some solvable extension $K/\Q$ such that for each $\ell$ the representation $\rho_\ell \simeq \Ind_{G_K}^{G_\Q} \sigma_\ell$ for a $2$-dimensional representation $\sigma_\ell$.
Then  all $\rho_\ell$ are automorphic. }

\bigskip

So far the automorphy of degree-$2d$ Scholl representations attached to spaces of weight-$\kappa$ noncongruence  cusp forms  is known systematically for $d=1$ and $\kappa \ge 2$ as a consequence of the Serre's conjecture established by Khare and Wintenberger \cite{KW09} and Kisin \cite{Ki9}, and certain cases (e.g. $\text{GO}_4$-type) of $d=2$ and odd $\kappa \ge 3$ obtained  by Liu and Yu in \cite{LY16}. Other known automorphy results are also for low degree  Scholl representations including   \cite{LLY,ALL,Long08, L5, HLV,  ALLL}.  Applying the results in this paper, we prove the potential automorphy for an infinite family of explicitly constructed Scholl representations with unbounded degrees, extending the automorphy results shown in \cite{LLY,ALL,Long08} alluded above.

For $n \ge 2$ denote by $\rho_{n,\ell}$ the $2(n-1)$-dimensional $\ell$-adic Scholl representation attached to the space of weight-$3$ cusp forms for the index-$n$ normal subgroup $\Gamma_n$ of the congruence subgroup $\Gamma^1(5)$ constructed in \cite{ALL}.  As explained in \S5.3, the representation $\rho_{n, \ell}$ decomposes as
$$ \rho_{n, \ell} = \bigoplus_{d|n, ~d>1} \rho_{d, \ell}^{new},$$
\noindent where $\rho_{n, \ell}^{new}$, as $\ell$ varies, form a compatible system of representations of $G_\Q$ of dimension $2\phi(n)$, and $\phi(n)$ is the degree of the cyclotomic field $\Q(\zeta_n)$ over $\Q$ with $\zeta_n$ a primitive $n$-th root of unity.
\bigskip

\noindent{\bf Theorem C.} (Theorem \ref{thm:potentialauto}) {\it For $n \ge 2$, there are $2$-dimensional $\ell$-adic representations $ \sigma_\ell $ of $G_{\Q(\zeta_n)}$ whose semi-simplifications form a compatible system, such that $\rho_{n,\ell}^{new} \cong \Ind_{G_{\Q(\zeta_n)}}^{G_\Q} \sigma_\ell$ for all primes $\ell$. For each $\ell$ the representation $\rho_{n,\ell}^{new}$ is potentially automorphic. Further, it is automorphic if either $n \le 6$ or $\sigma_\ell$ is potentially reducible.}
\bigskip

The (potential) automorphy results combined with the Atkin-Swinnerton-Dyer  congruences satisfied by the coefficients of noncongruence modular forms obtained in \cite{sch85b} provide a link between the Fourier coefficients of noncongruence modular forms with those of automorphic forms.  A much less expected consequence is illustrated in \cite{LL}, where the automorphy of a $2$-dimensional Scholl representation in turn shed new lights on the arithmetic properties of the associated $1$-dimensional space $S_\kappa$ of noncongruence cusp forms of integral weight $\kappa \ge 2$. More precisely, the automorphy is a key ingredient to prove that any cusp form in $S_\kappa$ with $\Q$-Fourier coefficients is a cusp form for a congruence subgroup if and only if its Fourier coefficients have bounded denominators. This result settles in part a long standing conjecture in the area of noncongruence modular forms.

This paper is organized as follows. Basic notation and definitions are given in Section \ref{ss:Preliminaries}, where potential automorphy results for degree-$2$ Galois representations and Clifford theory are reviewed.  Section \ref{sec: Scholl}  is devoted to Scholl representations and their 2-dimensional subrepresentations. {Our goal is to use advances on automorphy lifting theorem to prove Theorem A and its consequences. Owing to the irregularity of Scholl representations and their tensors in general, the approach in \cite{BGGT} cannot be applied directly. We pursue a variation by first using the results in  \cite{BGGT} to choose a suitable set $\CL$ of rational primes of Dirichlet density $1$; then for each $\ell\in \CL$ such that $\eta_\ell$ is not potentially reducible (the difficult case), we study properties of the reduction of $\eta_\ell$ and its symmetric square, which lead to the proof of Theorem A by applying the known potential automorphy results summarized in Theorem \ref{thm-known}.
In Section \ref{sec:GL1&2}, we study absolutely irreducible Scholl representations which are potentially of $\GL_1$- or $\GL_2$-type and prove Theorem B.   We end this section with a few potentially 2-isotypic examples of (potentially) automorphic Scholl representations attached to weight-2 and weight-4 cusp forms,} one of which  was originally computed by Oliver Atkin. Theorem C is proved in
 Section \ref{sec:wt-3} by realizing the Scholl representations $\rho_{n,\ell}$ as acting on the second \'etale cohomology of an elliptic surface $\E_n$ with an explicit model. When restricted to the Galois group of the cyclotomic field $\Q(\zeta_n)$, $\rho_{n,\ell}^{new}$ decomposes into the sum of a $2$-dimensional subrepresentation $\sigma_{n,\ell}$ and its conjugates by $\gal(\Q(\zeta_n)/\Q)$. Using the symmetries on $\E_n$ we determine the trace of $\sigma_{n,\ell}$ (Theorem \ref{thm:trace})  from which it follows that the semi-simplifications of $\sigma_{n,\ell}$ are compatible as $\ell$ varies. The information on $\Tr ~\sigma_{n,\ell}$ implies that the degree-$4$ representation $\tau_{n,\ell}$ of the Galois group of the totally real subfield of $\Q(\zeta_n)$ induced from $\sigma_{n,\ell}$ is $2$-isotypic over $\Q(\zeta_{2n})$ and Proposition \ref{prop-save} is then applied to conclude the potential automorphy of $\tau_{n,\ell}$ and hence $\rho_{n,\ell}$.  As a by-product of the trace computation, we obtain an estimate of certain character sums of Weil type, Corollary 5.5.  In \S \ref{ss:5.7} we show that $\tau_{n,\ell}$ admits QM over $\Q(\zeta_{2n})$, generalizing the known results for $n=3,4,6$ discussed in \cite{ALLL}. The QM structure provides an alternative approach to the potential automorphy of $\tau_{n,\ell}$ by appealing to results in \cite{ALLL}. Finally we remark that the same method can be used to obtain similar potential automorphy results for several infinite families of Scholl representations attached to weight-$3$ cusp forms for noncongruence subgroups of $\Gamma_1(6)$  constructed in the work of Fang et al \cite{L5}.

\bigskip

{\bf Acknowledgment.} The authors are grateful to the anonymous referee for the valuable suggestions and careful reading of this paper. They would also like to thank Michael Larson and Chia-Fu Yu for enlightening conversations and Fang-Ting Tu for her computational assistance. 

\section{Preliminaries}\label{ss:Preliminaries}
\subsection{Basic Notation}Let $F$ be a number field and $v$ a finite place of $F$. Denote by $F_v$ the completion of $F$ at $v$, $G_F:= \gal (\overline \Q/ F)$ the \emph{absolute Galois group} of $F$,  in which the absolute Galois group $G_{F_v}:= \gal (\overline F_v / F_v)$ of $F_v$ can be imbedded. The inertia subgroup  $I_v$ at $v$ consists of elements of $G_{F_v}$ which induce trivial action on the residue field of $F_v$. We use $\AF_v$ to denote the conjugacy class in $G_F$ of the \emph{arithmetic Frobenius} at $v$ which induces the Frobenius automorphism of the residue field of $F_v$, and use $\text{Frob}_v:=\AF_v^{-1}$ to denote  the \emph{geometric Frobenius} at $v$. In this paper, an \emph{$\ell$-adic Galois representation} of $G = G_F$ or $G_{F_v}$ is a continuous homomorphism from $G$ to the group of $\Qlbar$-linear automorphisms of a finite-dimensional $\Qlbar$-vector space.  We use  $\Qlbar$ (instead of a finite extension of $\Q_\ell$) as the scalar field for convenience.  By a \emph{character} we mean a $1$-dimensional representation. A  representation of $G_F$ or $G_{F_v}$ is said to be \emph{unramified} at $v$ if it is trivial on the inertia subgroup $I_v.$ % {\color{red} Let $\rho$ be an $\ell$-adic Galois representation of $G_F$ or $G_{F_v}$, we use $\ell$-adic Hodge theory to discuss property of $\rho$ restricted to $G_{F_v}$ for prime $v$ above $\ell$. The reader is referred to \cite{BL}  for basic notions in $\ell$-adic Hodge theory,  such as crystalline representations, de Rham representations and Hodge-Tate weights.
 %Call an $\ell$-adic representation of $G_F$ \emph{geometric} if it is unramified at almost all primes of $F$ and it is de Rham at  each prime of $F$ above $\ell$. %See for example \cite{BL} for more information.
% For the remainder of this paper, we always assume an $\ell$-adic Galois representation is geometric unless %otherwise specified.}

For brevity, we write  $\sigma \otimes \eta$ for the tensor  $\sigma \otimes_{\Qlbar } \eta$ of two $\ell$-adic representations $\eta$ and $\sigma$ of the same group. Suppose $F$ is Galois over $\Q$. Then for any representation $\rho$ of $G_F$ and $g \in G_\Q$, we can define another representation $\rho ^g$ of $G_F$, called the \emph{conjugate} of $\rho$ by $g$, via $\rho ^g (h): = \rho (g^{-1} h g)$ for all $h \in G_F$. It is easy to check that,  up to equivalence, $\rho ^g$ depends only on $g \in \gal (F/ \Q)$.

\subsection{$\GL_2$-type and potentially 2-isotypic}
\begin{defn} \label{def:GL2}Let $r$ be a positive integer. A continuous  semisimple  $\ell$-adic Galois representation $\rho_\ell$ of $G_\Q$  is said to be \emph{potentially of $\GL_r$-type} if there is a finite Galois extension $F/\Q$ such that the restriction  $\rho_\ell|_{G_F}$ decomposes into a direct sum of $r$-dimensional irreducible subrepresentations.
If, in addition, the $r$-dimensional subrepresentations are all isomorphic, $\rho_\ell$ is called \emph{potentially $r$-isotypic}. When $F=\Q$, one simply says $\rho_\ell$ is  \emph{of $\GL_r$-type} or \emph{$r$-isotypic} accordingly.

\end{defn}

In this paper we will be mostly concerned with  \emph{potentially of $\GL_2$-type} and \emph{potentially $2$- isotypic} representations.
As an example, consider the curve $C_{a,b} : y^2=x^6+ax^4+bx^2+1$ with $a, b\in \Q$. For generic choices of $a, b\in \Q$, it has genus 2. Let $\rho_{\ell,a,b}$  denote the  4-dimensional  $\ell$-adic Galois representation of $G_\Q$ arising from the Tate module of $\ell$-power torsion points on the Jacobian of $C_{a,b}$, which is known to be semi-simple. On $C_{a,b}$ there is the involution $\tau_1:(x,y)\mapsto (-x,y)$ defined over $\Q$. Its induced action on the representation space of $\rho_{\ell,a,b}$ decomposes the space into  eigenspaces $\sigma_{\ell,a,b,\pm}$ with eigenvalues $\pm 1$, both invariant under the Galois action. Therefore $\rho_{\ell,a,b} \cong \sigma_{\ell,a,b,+}\oplus \sigma_{\ell,a,b,-}$, and hence is of $\GL_2$-type. When $a=b$,  there is another  map $\tau_2: (x,y)\mapsto( \frac 1x,  \frac y{x^3})$ on $C_{a,a}$ defined over $\Q$.
It is straightforward to check that both $\tau_1$ and $\tau_2$ have order 2 and  $\tau_1\tau_2\tau_1^{-1}\tau_2^{-1}=
(\tau_1\tau_2)^2$ sends $(x,y)$ to $(x,-y)$ which induces the multiplication by $-1$ map on the Jacobian of $C_{a,a}$.  In other words, the induced actions of $\tau_1$ and $\tau_2$ on the Jacobian of $C_{a,a}$ anticommute with each other, thus $\tau_2$ intertwines the two representations $\sigma_{\ell,a,b,+}$ and $\sigma_{\ell,a,b,-}$ so that they are isomorphic 2-dimensional representations of $G_\Q$.
This makes $\rho_{\ell,a,a}$ a  prototype of
 $2$-isotypic representations.  Observe that $C_{a,a}$ is a two-fold cover of the elliptic curve
$E_a:y^2=x^3+ax^2+ax+1$ (assuming the discriminant of the curve is non-zero) which gives rise to a compatible system of $2$-dimensional $\ell$-adic representations $\sigma_{\ell,a}$ of $G_\Q$.  Thus $\sigma_{\ell,a}$ is isomorphic to  a subrepresentation of $\rho_{\ell,a,a}$. As $\rho_{\ell,a,a}$ is 2-isotypic, one concludes that $\rho_{\ell,a,a}\cong \sigma_{\ell,a}\oplus \sigma_{\ell,a}$. Later in Section \ref{ss:wt-2} we will see similar examples with $F$ being nontrivial extensions of $\Q$.

\subsection{Local properties and $\tau$-Hodge-Tate weights}{Let $F$ be a number field and $\rho$ be an $\ell$-adic representation of $G_F$. In this subsection we discuss  the local property of $\rho$ at a place $v$ of $F$ dividing $\ell$ via the $\ell$-adic Hodge theory. The reader is referred to \cite{BL}  for basic notions such as crystalline representations, de Rham representations, Hodge-Tate weights, etc. in $\ell$-adic Hodge theory.
Recall that  an $\ell$-adic Galois representation $\rho : G_F \to \GL_n (\Qlbar)$ is called \emph{geometric} if $\rho$ is unramified almost everywhere and $\rho|_{G_{F_v}}$ is de Rham for each prime $v$ of $F$ dividing $\ell$. For the remainder of this paper, we always assume an $\ell$-adic Galois representation is geometric unless otherwise specified.}

Each place $v$ of $F$ dividing $\ell$ corresponds to an embedding $\tau : F \to \Qlbar$, and $\tau$ extends to an embedding of $F_v$ in $\Qlbar.$
The \emph{$\tau$-Hodge-Tate weights} of $\rho$ is defined to be the multiset ${\rm HT}_\tau (\rho)$ consisting of the integers $i$ with multiplicity equal to
$$\dim _{\Qlbar} (\rho|_{G_{F_v}} \otimes_{\tau, F_v} \C_\ell (-i ))^{G_{F_v}} .$$
Here $\C_{\ell}$ denotes  the completion of $\Qlbar$ and $\C_\ell(-i)$ is the usual notation for Tate twists.

For example,  if $\rho =   \epsilon_\ell $ is the $\ell$-adic cyclotomic character, whose value at a geometric Frobenius element $\Frob_w$ at a finite place $w$ of $F$ not dividing $\ell$ is the inverse of the residual cardinality at $w$, then $\HT_\tau (\epsilon_\ell)= \{1\}$. When $F= \Q$, the trivial embedding $\tau$ will be omitted from  the notation $\HT_\tau$. Collected below are some useful facts concerning Hodge-Tate weights.

\begin{lemma}\label{lem:HT} Let $\rho$ and $\gamma$ be two $\ell$-adic Galois representations of $G_F$. Then \begin{enumerate}\item $\HT_\tau (\rho \otimes \gamma) = \{a_\tau+ b_\tau| a_\tau \in \HT_\tau (\rho), b_\tau \in \HT _\tau ({\g})\}.$
\item Suppose that $F/\Q$ is Galois. Then $\HT_{\tau} (\rho ^g) = \HT _{\tau g^{-1} } (\rho)$ for any $ g\in \gal (F/\Q)$.
\item  Suppose that $F$ is totally real and $\rho$ is a 1-dimensional geometric $\ell$-adic  representation. Then $\HT_\tau (\rho)$ is independent of the embedding $\tau : F \to \Qlbar$.
\end{enumerate}

\end{lemma}

\begin{proof}(1) It can be easily checked from the definition.

(2) Let $\tilde g \in G_\Q$ be a lift of $g$. We easily check that the map $\sum _i v_i \otimes a_i \mapsto \sum_i v_i \otimes \tilde g(a_i) $ induces a $\Qlbar$-linear bijection between $(\rho|_{G_{F_v}} \otimes_{\tau, F_v} \C_\ell (-i ))^{G_{F_v}}$ and $(\rho|_{G_{F_{g(v)}}} \otimes_{\tau, F_{g(v)}} \C_\ell (-i ))^{G_{F_g(v)}}$. Then the claim follows.

(3) Since $F$ is totally real by assumption, so is its maximal CM (complex multiplication) subfield. We apply the discussion before of Lemma A.2.1 of \cite{BGGT} to conclude that $\HT_\tau (\rho)$ (which is a singleton) does not depend on any embedding $\tau: F \to \Qlbar$.
\end{proof}

 \subsection{Compatible system of Galois representations and automorphy} \label{subsection:CS}
 \begin{defn}\label{def:compatiblefamily} Let $F$ be a number field.  For  each finite place $v$ of $F$, denote by $\t{rch}(v)$ the residual characteristic of $v$. Following \cite{BGGT}, by \emph{a rank $n$   compatible system of $\ell$-adic Galois representations $\R$ of $G_F$ defined over  $E$} we mean a quadruple
$$\{E, S , \{Q_\gp(X)\}, \{\rho_\lambda\}\}$$
where
\begin{itemize}
\item[1.] $E$ is a number field;
\item[2.] $S$ is a finite set of places of $F$;
\item[3.] for each finite place $\gp$ of $F$ outside $S$, $Q_\gp(X)$ is a monic degree $n$ polynomial in $E[X]$;
\item[4.] for each finite place $\lambda$ of $E$,
$$\rho_\lambda : G_F \longrightarrow \t{GL}_n(\overline{E}_\lambda)$$
is a continuous semi-simple representation such that
 \item $\rho_\lambda$ is unramified at finite places $\gp$ of $F$ outside $S$ with $\t{rch}(\gp) \ne \t{rch}(\lambda)$, and
$\rho _\lambda(\Frob_\gp)$ has characteristic polynomial $Q_\gp(X)$.

\end{itemize}
If we further assume
\begin{itemize}
\item $\rho_\lambda|_{G_{F_\gp}}$ is de Rham at finite places $\gp$ of $F$ with $\t{rch}(\gp) = \t{rch}(\lambda)$,  and further, it is crystalline if $\gp \not \in S$,
\end{itemize}
and the quadruple satisfies
\begin{itemize}
\item for each prime $\ell$ and each embedding $\tau: F \to \Qlbar $, there exists a fixed set $\tv_\tau$  of integers such that $\t{HT}_\tau (\rho_\lambda)= \tv_\tau$ for all places $\lambda$ of $E$ dividing $\ell$,
\end{itemize}
then we call the quintuple $\{E, S , \{Q_\gp[X]\}, \{\rho_\lambda\}, \{\tv_\tau\}\} $ a \emph{strongly compatible system}.
\end{defn}

We warn readers  that the above definition differs  from that in \cite{BGGT}: the strongly compatible system here is called  \emph{weakly compatible system} in \cite{BGGT}. Here we follow the terminology in the classical setting  by Serre in \cite{Serre}. The properties of $\rho_\lambda$ at primes above $\ell$  were added to the definition of compatible system in \cite{Serre} only in the recent decade. We adapt the above version which is convenient for our purposes. \bk
\begin{defn}\label{def:automorphic}
An $\ell$-adic Galois representation $\rho: G_F \to \t{GL}_n(\Qlbar)$ is called \emph{automorphic} if there exist  an isomorphism $\iota : \Qlbar \to \mathbb C$ and an automorphic representation $\pi\simeq  \otimes'_v \pi_v$ of $\GL_n (\mathbb A_F)$ such that for almost all primes $v$ of $F$ the (Frobenius-semi-simplification of the) Weil-Deligne representation associated to $\rho|_{G_{F_v}}$ via $\iota$ is isomorphic to the Weil-Deligne representation associated to $\pi_v$  as described by the local Langlands correspondence. (See, for instance, \cite{BGGT} for details.)
In particular,
$\iota$ sends the eigenvalues of the characteristic polynomial of $\rho(\Frob_v)$ to the Satake parameters of $\pi_v$ for almost all places $v$ of $F$. We call $\rho$ \emph{potentially automorphic} if there exists a finite extension $F'$ of $F$ so that $\rho|_{G_{F'}}$ is automorphic.
\end{defn}

Let $F'$ be a soluble extension of $F$. According to the  solvable base change in \cite{AC89} by Arthur and Clozel, if a representation $\rho$ of $G_F$ is automorphic, then so is its restriction $\rho|_{G_{F'}}$.  Conversely, if a representation $\sigma$ of $G_{F'}$ is automorphic,  then so is the induced  representation $\text{Ind}_{G_{F'}}^{G_F} \sigma$ of $G_F$.

\subsection{ Potential automorphy results for degree-$2$ Galois representations.}\label{subsec-add} We summarize the known (potential) automorphy results for $2$-dimensional Galois representations which will be used later in the paper. Let $F$ be a totally real field and $\eta: G_F \to \GL_2(\Qlbar)$ an $\ell$-adic Galois representation. We call $\eta$ (totally) \emph{odd} if $\det (\eta) (c) = -1$ for any complex conjugation $c\in G_F$.
For any finite-dimensional $\ell$-adic representation $\sigma$ of $G_F$, its ambient space always contains a $G_F$-stable $\O_{\Qlbar}$-lattice $L$. Let $\mathfrak m$ be the maximal ideal of $\O_{\Qlbar}$.  Then $\sigma$ induces an action $\bar \sigma$ of $G_F$ on the quotient space $L/\mathfrak m L$ over the residue field $\overline \F_\ell$, called  the \emph{reduction} of $\sigma$.
It is well-known that the semi-simplification of $\bar \sigma$ does not depend on the choice of the lattice $L$.

\begin{theorem}\label{thm-known} Let $F$ be a totally real field and $\eta: G_F \to \GL_2(\Qlbar)$ be a  continuous representation. Assume the following:
\begin{itemize}
\item [(a)]$\eta$ is irreducible and unramified almost everywhere;
\item [(b)] $F$ is unramified at $\ell$;  for each prime $v$ of $F$ above $\ell$, $\eta|_{G_{F_v}}$ is crystalline and for each embedding $\tau: F \to \Qlbar$, $\t{HT}_{\tau}(\eta)=\{a_\tau, a_\tau +b_\tau\}$ with $0 < b _\tau <(\ell-1)/2$;
\item [(c)] $\sym^2 \bar \eta|_{G_{F(\zeta_\ell)}}$ is irreducible;
\item [(d)] $\ell >7$.
\end{itemize}

\noindent Then the following statements hold:
\begin{enumerate}
\item There is a finite totally real Galois extension $F'/F$ such that $\eta|_{G_{F'}}$ is automorphic.

\item Suppose that the $2$-dimensional $\ell$-adic  representations $\eta_i, i= 1, \dots , m,$ of $G_F$ satisfy the assumptions (a)-(d).  Then there exists a finite totally real Galois extension $F'/F$ such that $\eta_i |_{G_{F'}}$ is automorphic for all $1 \le i \le m$.

\item If $F=\Q$, then $\eta$ is automorphic (modular).

\end{enumerate}
\end{theorem}
\begin{proof} The assumption (c) implies the irreducibility of $\bar\eta|_{G_{F(\zeta_\ell)}}$.
Together with the assumption (d) we conclude that $\eta$ is totally odd from \cite[Prop.2.5]{Cal11}. { The
assumption (b) implies that $\eta|_{G_{F_v}}$ is potentially diagonalizable so that a required condition to apply  Theorem C of \cite{BGGT} is satisfied.  Together with Lemma 1.4.3 (iii) in \cite{BGGT}, we see that} the assertion (1) is the special case of Theorem C in \cite{BGGT} for $n=2$. %{\color{red} together with Lemma 1.4.3 (iii) in \cite{BGGT} %(assumption (b) then implies that $\eta|_{G_{F_v}}$ is potentially diagonalizale which is required by Theorem C)} .

For (2), we first warn the reader that this is not a  (formal) consequence of (1),  because it is not known that the automorphy of $\eta|_{G_{F'}}$ would imply the automorphy of $\eta|_{G_{F''}}$ for \emph{any} finite totally real extension $F''/F'$. To prove (2), we use \cite[Thm. 4.5.1]{BGGT} and the idea of the proof of Corollary 4.5.2 \emph{loc. cit.} Pick a totally imaginary quadratic extension $M/ F$ which is linearly disjoint from $K(\zeta_\ell)$ over $F$, where $K$ is a finite extension contained in all splitting fields of $\sym ^2 \bar \eta_i$, namely the field fixed by the kernel of $\sym ^2 \bar \eta_i$. As observed above, each  $\eta_i$ is totally odd.
Then $(\eta_i|_{G_M}, \det \eta_i ) $ satisfies the assumption of Theorem 4.5.1 in \cite{BGGT} so that  there exists a finite Galois CM extension $M_1/ M$ such that $(\eta_i|_{G_{M_1}},  \det \eta_i|_{G_{F'}} )$ is automorphic for all $i$, where
$F'$ is the maximal totally real subfield of $M_1$.  Then $\eta_i|_{G_{F'}}$ is also automorphic by Lemma 2.2.2 \emph{loc. cit.} for all $i$.

If $F=\Q$, then the main result in \cite{DFG} together with the input of Serre's conjecture and oddness of $\eta$ implies that $\eta$ is modular.
This proves (3).
\end{proof}

\subsection{Clifford theory in the context of Galois representations}\label{sec:clifford} We end this section by  summarizing some useful results in \cite{Clifford37}  in the context of Galois representations. Let $F$ and $k$ be fields. Denote by $$\pi: \t{GL}_n (k ) \twoheadrightarrow \t{PGL}_n (k)$$ the natural projection.  Two Galois representations $ \tau, \tau' : G_F \to \t{GL}_n (k)$ are said to be \emph{projectively  equivalent} if there exists an invertible matrix $A \in \t{GL}_n (k)$ such that $$\pi \circ \tau = \pi \circ (A \tau' A^{-1}).$$
This is equivalent to the existence of a character $\chi: G_F \to k ^\times$ so that $\tau \simeq \chi \otimes \tau'$. If $\chi$ is a character of finite image, then $\tau$ and $\tau'$ are called \emph{finitely projectively equivalent.}

The next useful  Theorem follows from \cite{Clifford37} and Tate's result on the vanishing of Galois cohomology \cite{Tate-62-ICM} or \cite[\S6.5]{Tate}.

 Let $k$ be {an algebraically closed field} and $\rho : G_F \to {\rm GL}_n(k)$ be an irreducible representation.  Given a  finite Galois extension  $L/F$, decompose the restriction of $\rho$ to $G_L$ into a direct sum of irreducible representations of $G_L$:
$$ \rho|_{G_L} \simeq \sigma _1 \oplus \cdots \oplus \sigma _m. $$
Write $\sigma$ for $\sigma_{ 1}$ and set $H := \{g \in G_F | \sigma^g \simeq \sigma \}.$ Denote by $M : = \overline \Q^H$ the fixed field of $H$.  Note that $H$ contains $G_L$ and $M$ is a subfield of $L$ containing $F$.
\begin{theorem}\label{thm:Clifford}
Under the above setting, the following statements hold:
\begin{enumerate}
\item For each $i = 1, \dots , m$ there exists an element $g(i) \in \gal(L/F)$ such that $\sigma _i \simeq \sigma^{g(i)}$. Consequently the $\sigma_i$'s have the same dimension and $[M:F]$ is equal to the number of non-isomorphic  $\sigma_i$'s.
\item There exist representations $\eta: G_M \to  {\rm GL}_r(k)$ and $\gamma: G_M \to {\rm GL}_s (k)$ such that
\begin{itemize}
\item [(2a)] $\eta|_{G_L}$ is finitely  projectively equivalent to $\sigma$, and $\gamma$ has finite image such that  $\gamma |_{G_L}$ is finitely projectively equivalent to  $s$ copies of the trivial
representation of $G_L$.
\item [(2b)] $\rho \simeq \Ind_{G_M}^{G_F} (\gamma \otimes \eta )$.
\end{itemize}
\end{enumerate}

\end{theorem}
 For the sake of self-containedness, we sketch a (slightly different) proof here.  Let $V$ denote the underlying $k$-space of $\rho$. Define $W$ to be the subspace of $V$ spanned by $\sigma ^g  $ for all $g \in H$. Set $H':= \{g \in G_F| g (W) = W\}$. Obviously, $W$ is a representation of $H'$. Theorem 2 in \cite{Clifford37} shows that $V = \Ind^{G_F}_{H'}W$. Now we prove that $H'= H = G_M$. Obviously, $H \subset H'$ by definition. For any $g \in H'$, $\sigma ^g (W) \subset W$ by definition. But $W$ is a direct sum of irreducible representations isomorphic to  $\sigma$. So $\sigma ^g \simeq \sigma$ and hence $g \in H$. Therefore $V = \Ind^{G_F}_{G_M}W$.  Since $W|_{G_L}$ is $r$-isotypic by the construction of $W$, Theorem 3 in \cite{Clifford37} shows that there exist projective representations $\tilde \gamma: G_M \to \t{PGL}_s (k)$   of $G_M/G_L$,
 and $\tilde \eta  : G_M \to \t{PGL}_r (k)$ so that $W \simeq \tilde \gamma \otimes \tilde \eta$, where $\tilde \gamma |_{G_L}$ is projectively equivalent to $s$ copies of the trivial representation of $G_L$
and
$\tilde \eta |_{G_L}$ is projectively equivalent to $\sigma$. By Tate's result on vanishing of Galois cohomology, $\tilde \gamma$ admits a lifting $ \gamma : G_M \to \t{GL}_s (k)$ with finite image. Hence $\tilde  \eta $ also has a lifting $\eta : G_M \to \t{GL}_r (k)$ such that $W \simeq \gamma \otimes \eta$. This proves the theorem.

 \section{Galois representations arising from noncongruence cusp forms and their $2$-dimensional subrepresentations} \label{sec: Scholl}
 \subsection{Scholl representations associated to noncongruence cusp forms} Let $\G \subset \t{SL}_2(\Z)$ be a \emph{noncongruence} subgroup, that is, $\G $ is a finite index subgroup of $ \t{SL}_2(\Z)$ not containing any principal congruence subgroup $\G(N)$. For any integer $\kappa \ge 2$,   the space $S_\k(\G)$ of weight $\k$ cusp forms for $\G$ is finite-dimensional;  denote by $d = d(\G, \k)$ its dimension. Assume that the compactified modular curve $\G \backslash\mathfrak H^*$ (by adding cusps) is defined over $\Q$  and the cusp at infinity is $\Q$-rational.  For even $\k \ge 4$ and any prime $\ell$, in \cite{sch85b}  Scholl constructed an $\ell$-adic Galois representation $\rho_\ell: G_\Q \rightarrow \t{GL}_{2d}(\Q_\ell)$ attached to $S_\k(\G)$. It turns out that there exist a  finite set $S$ of primes, polynomials $Q_p(X) \in \mathbb Z[X]$ for $p \notin S$, and a finite set $\tv$ so that  $\{ \Q, S , \{Q_p (X) \},  \{\rho_\ell\}, \tv \}$ form a \emph{strongly  compatible system}\footnote{It is unclear that $\rho_\ell$ is always semi-simple from Scholl's construction. So in the following, we always replace $\rho_\ell$ by its semi-simplification if needed.}. Here $\tv = \HT(\rho_\ell)= \{0, \dots , 0, 1-\k , \dots , 1-\k\}$ consists of $1-\k$ and $0$, each with multiplicity $d$, and is independent of $\ell$.
 Scholl also showed that all the roots of $Q_p(X)$
have the same complex absolute value $p^{(\k-1)/2}$ (cf. \S5.3 in \cite{sch85b}).
These results of
Scholl can be extended to odd weights under some
extra hypotheses (e.g., $\pm(\G \cap \G(N))= \pm (\G) \cap \pm
(\G(N))$ for some $N\ge 3$, where $\pm : \t{SL}_2 (\Z) \to \t{PSL}_2(\Z)$ is the
projection). The reader is referred to the end of \cite{sch85b}
for more details. In this paper we
assume that $\rho_\ell$ exists.

 Let $V_\ell$ denote the underlying $\Q_\ell$-space of $\rho_\ell$. There exists a perfect, Galois action compatible pairing
 $$V_\ell {\times} V_\ell \to \Q_\ell  (-\k +1)$$which is alternating (resp. symmetric) when $\k$ is even (resp. odd).
In particular, we have $\rho ^\vee_\ell \simeq \epsilon_\ell ^{\k-1} \rho _\ell$ for the dual representation  $\rho ^\vee_\ell$ of $\rho_\ell$,  where $\epsilon_\ell $ denotes the $\ell$-adic cyclotomic character.

For the remainder of the paper, we reserve $\rho_\ell$ for the $\ell$-adic Galois representation associated to a noncongruence subgroup and call it a \emph{Scholl representation} if no confusion arises.  As explained in the Introduction, Scholl representations are expected to correspond to certain automorphic forms. But since they are irregular when $d>1$, the currently known (potential) automorphy lifting theorem can not be applied directly. In this section, we will show that Scholl representations potentially of $\GL_2$-type are (potentially) automorphic. See Theorem \ref{thm-try} for the precise statement.

A (general) representation $\sigma$ of a Galois group $G_F$  is said to be \emph{potentially reducible} if there exists a finite index subgroup $H$ of $G_F$ such that $\sigma|_H$ is reducible; otherwise, it is called \emph{strongly irreducible}. Recall that  a $2$-dimensional representation $\sigma$ of the Galois group of a totally real field $F$ is \emph{totally odd} if $\det \sigma (c) = -1$ for any complex conjugation $c \in G_F$.

 The  goal of \S \ref{sec: Scholl} is to prove

\begin{theorem}\label{thm-try} Let $F$ be a totally real field and $\eta_\ell : G_F \to \GL _2 (\Qlbar)$ be a system of 2-dimensional $\ell$-adic Galois representations. Suppose there exist a finite extension $K$ of $F$ and a compatible system of Scholl representations $\{\rho_\ell\}$ such that for each $\ell$,  $ \eta_\ell | _{G_K}$ is a  subrepresentation of $\rho_\ell|_{G_K}$.  Then there exists a set $\CL$ of {rational} primes of Dirichlet density 1 so that for each $\ell \in \CL$, $\eta_\ell$ is totally odd
and potentially automorphic. Moreover, $\eta_\ell$ is automorphic if $F=\Q$ or $\eta_\ell$ is potentially reducible.
\end{theorem}

The proof will be divided into two parts according as $\eta_\ell$  is potentially reducible or strongly irreducible.

\subsection{Potentially reducible $\eta_\ell$ in Theorem \ref{thm-try}}

We begin by exploring general properties of $\eta_\ell$ in Theorem \ref{thm-try}, including its determinant, irreducibility, and Hodge Tate weights.

\begin{prop}\label{prop-purity} Let $F$ be a totally real field and $\eta_\ell: G_F \to \GL_2 (\overline \Q_\ell)$ be a Galois representation.  Suppose that
there exists a finite extension $K/F$  so that $\eta_\ell |_{ G_K}$ is isomorphic to a subrepresentation of $ \rho_\ell |_{G_K}$ for  a Scholl representation $\rho_\ell$  of $G_\Q$ associated to a space of cusp forms of weight $\k\ge 2$. Then $\eta_\ell$ is {(absolutely)} irreducible, $\det \eta_\ell = \epsilon_\ell^{1-\k} \chi$ for a character $\chi$ of finite order, and $\HT_\tau (\eta_\ell) = \{0 , -\k +1\}$ for all embeddings $\tau : F \to \Qlbar$.
\end{prop}

\begin{proof}
We first prove the statement on $\det \eta_\ell$ and Hodge-Tate weights.  Let $\tau$ be an embedding of $F$ into $\Qlbar$. Note that $\HT_\tau(\eta_\ell) = \HT_{\tau'}(\eta_\ell |_{G_K})$ for any embedding $\tau' : K \to \Qlbar$ extending $\tau$.
Since $\eta_\ell|_{G_K}$ is a subrepresentation of $\rho_\ell|_{G_K }$, we see that  $\HT_\tau( \eta_\ell)$ only has 3 possibilities: $ \{0, 0\}$, $\{-\k+1, -\k+1\}$, or $\{0, -\k+1\}$.
Then the  determinant of $\eta_\ell$ has $\HT_\tau(\det \eta_\ell) = \{r\}$ with $r = 0$ or $2-2\k$ or $1-\k$,  equal to the sum of the two weights of $\eta_\ell$ accordingly. Since $F$ is totally real, by Lemma \ref{lem:HT}, $\HT_\tau (\det \eta_\ell)$ is the same for all $\tau$. Hence
$\det \eta_\ell = \epsilon_\ell ^{r} \chi$ for some character $\chi$ with Hodge-Tate weight $0$. Since $\chi$ has Hodge-Tate weights $0 $ for all primes $\gp| \ell$, we have that $\chi (I_\gp)$,  the image of the inertia group at $\gp$,  is a finite group, by Proposition 3.56 of \cite{FO}. For any finite prime $\gp \nmid \ell$, $\chi (I_\gp)$ is also a finite group. This is because the wild ramification group must have a finite image, and the same holds for the tame ramification group, resulting from the relation between a Frobenius element and a generator of the tame ramification group and the fact that $\chi$ is a character. As $\chi$ is ramified at finitely many places, we conclude the existence of a positive integer $n$ such that $\chi^n$ is a character unramified at \emph{all} places. By class field theory $\chi^n$ has finite order, therefore so has $\chi$.

As a subrepresentation of $\rho_\ell |_{G_K}$, at almost all finite places $\gp$ of $K$, the roots of the characteristic polynomial of  $\eta_\ell(\Frob_\gp)$ have the same complex absolute value $q^{\frac{\k-1}{2}}$, where $q$ is the cardinality of the residue field of $\gp$. Hence the complex absolute value of $\det \eta_\ell (\Frob _\gp)$ is $q^{\k-1}$. On the other hand, the complex absolute value of $\epsilon_\ell (\Frob_\gp)$ is $q^{-1}$ and that of $\chi (\Frob _\gp)$ is $1$ since $\chi$ has finite order so that $\det \eta_\ell= \epsilon_\ell ^{r} \chi$ at $\Frob_{\gp}$ has complex absolute value $q^{-r}$. This proves that $r = 1-\k$ and $\HT_\tau (\eta_\ell) = \{0 , -\k+1\}$ for all $\tau$.

Now suppose that  $\eta_\ell$ is reducible. Then the semi-simplification of $\eta_\ell $ is $\zeta \oplus \zeta'$ for some $1$-dimensional $\lambda$-adic  representations $\zeta$ and $\zeta'$ of $G_F$.  By Lemma \ref{lem:HT} again  and replacing $\zeta$ by $\zeta'$ if necessary, we may assume that $\HT_\tau (\zeta) = 0$ for all  $\tau$ and  $\HT_{\tau} (\zeta ') = 1-\k$  for all $\tau $. The above argument  shows that $\zeta = \chi'$ and $\zeta'=  \epsilon_\ell ^{1-\k} \chi'' $ for some finite order characters $\chi'$ and $\chi'' $. The characteristic polynomial  of $\eta_\ell(\Frob_\gp)$ at an unramified finite place $\gp$ of $K$ is $(X- \chi '(\Frob_\gp)) (X- \epsilon_\ell^{1-\k} (\Frob_\gp) \chi'' (\Frob_\gp))$  with two roots of different complex absolute values, contradicting Scholl's result.   So $\eta_\ell$  must be irreducible.
\end{proof}

Next we show that if $\eta_\ell$ in Theorem \ref{thm-try} is potentially reducible, then it is totally odd and automorphic.

\begin{prop}\label{prop-auto}Suppose that $\eta_\ell$ in Proposition \ref{prop-purity} is potentially reducible. Then $\eta_\ell$ is totally odd, and there is a quadratic CM extension field $M$ of $F$ and a character $\chi_1$ of $G_M$ such that $\eta _\ell= \Ind^{G_F}_{G_M}\chi_1$. Consequently $\eta_\ell$ is automorphic.
\end{prop}

\begin{proof} It follows from the proposition above that $\eta_\ell$ is irreducible with $\HT_\tau (\eta_\ell) = \{0 , -\k +1\}$ for all embeddings $\tau : F \to \Qlbar$. By assumption, there is a finite extension $L$ of $F$ such that $\eta_\ell |_{G_L}$ is reducible. Then by Theorem \ref{thm:Clifford},   $\eta_\ell |_{G_L}\cong \chi_1 \oplus \chi _2$ with distinct characters  $\chi_i$ of $G_L$, and furthermore, there exists a  quadratic extension $M$ of $F$ so that $\chi_1$ can be extended to a character of $G_M$ and  $\eta_\ell = \Ind^{G_F}_{G_M}\chi_1$.
It turns out that $M$ has to be a CM field, for  otherwise $\chi_1 $ would be a  power of the $\ell$-adic cyclotomic character twisted by some finite character (see Proposition 1.12 of \cite{Fargues}) and this contradicts the fact that $\eta_\ell$ has two distinct Hodge-Tate weights $\{0, 1-\k\}$. This shows that  $\eta_\ell$ is totally odd. As $\chi_1$  is geometric, it is well-known that $\chi_1$ is automorphic  (see for example \cite{Fargues}), hence so is $\eta_\ell$  by quadratic automorphic induction  \cite{AC89}.
\end{proof}

\subsection{A proof of Theorem \ref{thm-try}}\label{ss:3.3}

In view of the previous subsection, the heart of the proof of Theorem \ref{thm-try} is to handle the strongly irreducible $\eta_\ell$'s. Owing to the irregularity of Scholl representations and their tensors in general, the approach in \cite{BGGT} by Barnet-Lamb, Gee, Geraghty,  and Taylor cannot be applied directly. We pursue a variation as follows. First, using the results in  \cite{BGGT} we choose a suitable set $\CL$ of rational primes of Dirichlet density $1$ with certain properties. Then for each $\ell\in \CL$ such that $\eta_\ell$ is strongly irreducible, we study properties of the reduction $\bar \eta_\ell$  and its symmetric square $\sym ^2 \bar \eta_\ell$, which enable us to conclude  Theorem \ref{thm-try} by applying Theorem \ref{thm-known}.  Unless specified otherwise, the representations are over algebraically closed fields for convenience. Hence there is no distinction between irreducibility and absolute irreducibility.

The set $\CL$ arises from applying results in \S5.2 of \cite{BGGT} for a number field $F$ (not necessarily totally real) as follows. Let $\{ \Q, S , \{Q_\gp (X)\}, \{r_\ell\}, \{\tv_\tau\}\}$ be a  strongly  compatible system of  representations of $G_F$  as in Definition \ref{def:compatiblefamily}.
Assume that, after conjugating by some $g_\ell \in \GL_n (\Qlbar)$, $r_\ell (G_F) \subset \GL_n (\Q_\ell)$ for all $\ell$. {\footnote{Note that $r_\ell (G_F) \subset \GL_n(\Qlbar)$ by the definition of a compatible system.  It is not clear that in general we can find $g_\ell \in \GL_n (\Qlbar)$ so that $r_\ell (G_F) \subset \GL_{n}(\Q_\ell) $ after conjugating by $g_\ell$, although the characteristic polynomial of each  $h \in G_\Q$ has coefficients in $\Q_\ell$. Luckily this is true for $r_\ell= \rho_\ell$ or $r_\ell = \rho_\ell \otimes \rho_\ell$ used below.}
Let $V_\ell$ denote the $\Q_\ell$-ambient space of $r_\ell$ with dimension $n$.
Let $G_\ell $ be the Zariski closure of {$r_\ell(G_F)$} inside $\GL_n(\Q_\ell)$
with the identity component $G^\circ_\ell$, $G_\ell^{\t{ad}}$ the quotient of $G_\ell^\circ$ by its radical, and $G_\ell^{\t{sc}}  $ the simply connected cover of $G_\ell^{\t{ad}}$. Then we get maps
\begin{equation}\label{Eqn-reductivegps}
\xymatrix{{G^\circ _\ell} \ar@{->>}[r] ^\sigma  & G_\ell^{\t{ad}}  &  G_\ell ^{\t{sc}} \ar@{->>}[l] _ {\tau}} .
\end{equation}
As in the beginning of \S 5.2 in \cite{BGGT},  let $Z_\ell$  denote the center of $G_\ell^\circ$ and $H_\ell:= G _\ell^\t{sc} \times Z_\ell$.
Then there is a natural surjection of algebraic groups $H_\ell \twoheadrightarrow G_\ell^\circ$ with a finite and central kernel. So the  $G_F$-action on the ambient space $V_\ell$ of $r_\ell$ induces a representation of $G_\ell^{\t{sc}}$ on $V_\ell$. In particular, if $\sigma_\ell$ is a subrepresentation of $ \overline \Q_\ell \otimes_{\Q_\ell}  r_\ell$, then the $G_\ell^{\t{sc}}$-action on $V_\ell$ leaves invariant the ambient space of $ \sigma_\ell$ (contained in $ \Qlbar  \otimes_{\Q_\ell}  V_\ell $).  Now apply Proposition 5.2.2 of \cite{BGGT} (where no regularity condition,  i.e. Hodge-Tate weights being distinct,  is required) to our compatible system $\{ r_\ell = \rho_\ell \otimes_{\Q_\ell} \rho_\ell \}$ restricted to $G_K$. Then we see that $ G _\ell ^{\t{sc}}$ acts on the ambient space $W_\ell$ of $\t{Sym} ^2 \eta_\ell$. Furthermore  Proposition 5.2.2 of \cite{BGGT} shows that there exists a  subset $\CL$ of rational primes of Dirichlet density 1 with the following
properties: for any $\ell \in \CL$, there exists a semi-simple group scheme  $\wt G_\ell^{\t{sc}}$ over $\Z_\ell$ with generic fiber $G_\ell ^\t{sc}$ so that $ \wt G_\ell ^\t{sc} (\Z_\ell) = \tau ^{-1} ( \sigma ( r_\ell (G_F)   \cap G^\circ _{\ell})), $ where $\sigma$ and $\tau$ are maps described in \eqref{Eqn-reductivegps}.  Also there exists a $\Z_\ell$-lattice $\Lambda$ inside $ V_\ell$ so that the actions of $G_\ell ^\t{sc}$  and $G_K$ on $V_\ell $ can be naturally extended to  $\wt G_\ell^{\t{sc}}$- and $G_K$-actions on $\Lambda$.

Removing finitely many primes from $\CL$ if necessary, we further require that each $\ell \in \CL$ satisfies the following three conditions:

(i) $K$ is unramified above $\ell$,

(ii) $\ell - 1 > 6 \k$,

(iii) $\rho_\ell |_{G_{\Q_\ell}}$  is crystalline. This follows from the fact that $\rho_\ell$ forms a strongly compatible system as proved by Scholl.
\medskip

 Next we describe some properties of strongly irreducible $\eta_\ell$ with $\ell \in \CL$.

\begin{lemma}\label{lem-shapeofH} Let $\ell \in \CL$. Write $U_\ell$ for the ambient space of $\eta_\ell$ in Theorem \ref{thm-try}.
 Suppose that  $\eta_\ell$ is  strongly irreducible.  Then the $G _\ell^\t{sc}$-action on $U_\ell$ factors through the action of $\t{SL}_2$. In particular, the $G_\ell ^\t{sc }$-action  on the ambient space $W_\ell$ of $\t{Sym} ^2 \eta_\ell$ is  irreducible.
\end{lemma}
\begin{proof} Denote by $N$ the identity component of the Zariski closure of $\eta_\ell(G_K)$ in $\GL_2(\overline \Q_\ell)$ which  acts on $U_\ell$.
The projective image of $N$ in  $\t{PGL}_2$ is a connected subgroup, which we claim to be the whole group.
For this, it suffices to look at the Lie algebra $\t{Lie} (\t{PGL} _2)$ by Proposition 3.22 of Milne's note \cite{Mil}. Since $\t{Lie} (\t{PGL} _2)$ consists of  $2 \times 2$ matrices with trace 0,  one finds that the  nontrivial connected algebraic subgroups of $\t{PGL}_2$, up to conjugation, have 4 possibilities: the split torus, the unipotent subgroup,  the Borel subgroup, and $\t{PGL}_2$ itself. Since the Zariski closure of $\eta_\ell(G_K)$ has finitely many connected components, the first $3$ cases imply the reducibility of $\eta_\ell |_{G_M}$ for some finite extension $M$ over $K$, contradicting the assumption on $\eta_\ell$ being strongly irreducible. So the projective image of $N$ is $\t{PGL}_2$. It is easy to check that $\t{SL}_2 $ is contained in the commutator subgroup $[\t{PGL}_2, \t{PGL}_2]=[N, N ]$ of $N$. Hence $\t{SL}_2 \subset N$. Finally, since $\det (\eta_\ell)= \chi \epsilon_\ell^{1-\k}$ with $\chi$  being  a finite character by Proposition \ref{prop-purity} and $\k \ge 2$, so $N/\t{SL}_2$ must have dimension 1. This shows that  $N$ has dimension $4$. Since $N$ is connected, $N$ must be $\GL_2$. This shows that the  $G_\ell ^\t{sc}$-action on $U_\ell$ factors through the action of $\GL_2 ^{\t{sc}}= \t{SL}_2$, and then the $G^\t{sc}_\ell $-action on  $W_\ell = \t{Sym}^2 U_\ell$ is  irreducible.
\end{proof}

\begin{prop}\label{prop-residue-irreducible}  Let $\ell \in \CL$. If $\eta_\ell$ in Theorem \ref{thm-try} is strongly irreducible,   then  the reduction of $\t{Sym} ^2 \eta_\ell |_{G_{K (\zeta_\ell)}}$ is  irreducible.
\end{prop}

\begin{proof} We first remark that Proposition 5.3.2 in \cite{BGGT} cannot be  applied directly because $\rho_\ell \otimes \rho_\ell$ may not be regular.  Fortunately, we can follow their idea, which is built upon the work \cite{Lar95} of Larsen.

By Lemma \ref{lem-shapeofH}, if $\eta_\ell|_{G_M}$ is  strongly irreducible, then $\t{Sym} ^2 \eta_\ell$ is an absolutely irreducible $G^\t{sc}_\ell$-module.  Then, by Proposition 5.3.2 (6) of \cite{BGGT}, there exists a finite unramified extension  $M _\lambda$ of $\Q_\ell$ so that $\t{Sym}^2 \eta_\ell$ is defined over $M_\lambda$ as an absolutely irreducible $\wt G_\ell ^{\t {sc}}$-module. Finally, the mod  $\lambda$ reduction of $(\O_{M_\lambda} \otimes_{\Z_\ell} \Lambda ) \cap \t{Sym} ^2 \eta_\ell$ is absolutely irreducible as an $\wt G _\ell ^{\t{sc}} (\Z_\ell)$-module. Therefore, we conclude that  as a $G_K$-module, the reduction of $ \t{Sym} ^ 2  \eta_\ell $ is absolutely irreducible.

 It remains to prove that $\t{Sym} ^2 \bar \eta_\ell |_{G _{K (\zeta_ \ell)}}$ is irreducible. By construction, $K$ is unramified over $\ell$ so that $\gal (K(\zeta_\ell)/K) \simeq (\Z/ \ell \Z) ^\times$. Write $\overline W _\ell$ for the ambient space of $\t{Sym} ^2 \bar \eta_\ell$. Suppose $\overline  W_\ell|_{G_{K(\zeta_\ell)}}$ is reducible and we will derive a contradiction.   Since $\overline W_\ell$ is irreducible, by Clifford theory recalled in \S \ref{sec:clifford}, we see that $\overline W_\ell|_{G_{K(\zeta_\ell)}}$ is a  direct sum of characters $\chi _i$ of $G _{K(\zeta_\ell)}$. Let $\chi = \chi _1$ and set $H := \{g \in G_K | \chi ^g \simeq \chi \} $ and $M = (\overline \Q) ^H \subset K(\zeta_\ell)$.  By Theorem \ref{thm:Clifford},  $\overline  W_\ell \simeq \t{Ind} _{G_M} ^{G _K} (\chi' \otimes \gamma)$ for a character $\chi'$ of $G_M$ extending $\chi$ and a representation $\gamma$ of $G_M$ with finite image. Since $\overline  W_\ell$ has dimension 3, either   $M=K$ or $[M: K] =3$.

Case 1:  $M = K$. In this case, we have  $\overline  W_\ell \simeq \chi' \otimes \gamma$ where $\chi'$ is a character of $G_K$ extending $\chi$. Then $\overline  W'_\ell := ((\chi') ^{-1} \otimes \overline W_\ell) |_{G_{K(\zeta_\ell)}} $ is trivial. So $\overline  W'_\ell$ is a 3-dimensional representation of the cyclic group $\gal (K(\zeta_\ell)/ K)$ and hence must be reducible, so is $\overline  W_\ell$, a contradiction.

Case 2: $[M: K] = 3$. In this case, $\chi$ can be extended to a character of $G_M$ so that $\overline  W_\ell \simeq  \t{Ind}_{ G_M} ^{G_K} \chi $ and $M$ is the unique subfield of $ K(\zeta_\ell)$ with degree $3$ over $K$. We have $\ell -1 > \k$. Let $v$ be a prime of $K$ above $\ell$ with $[K_v :\Q_\ell] = f$ and $I_v$ be the inertia subgroup of $G_v: = \gal (\overline K_v/ K_v)$. It follows from the Fontaine-Lafaillle theory that $\bar \eta_\ell|_{I_v} \simeq \omega^h_{2f} \oplus \omega ^{h\ell}_{2f}$ (resp. $(\bar \eta _\ell |_{I_v} )^{\t{ss}}= \omega^a_f \oplus \omega ^b_f $) if $\bar \eta_\ell|_{G_v}$ is irreducible
 (resp. reducible).  Here $\omega_m$ denotes the fundamental character given by \begin{equation}\label{eq:fund-cha}\omega_{m} (g) = \frac{g (\sqrt[\ell ^{m}-1]{\ell })}{\sqrt[\ell ^{m}-1]{\ell }}\end{equation}  for $g \in G_K$, and  $h = \sum\limits_{i = 0} ^{2f-1} h_i p ^i$ with $\{h_i, h _{i +f}\}= \{0,  1-\k\}$ (resp. $a = \sum\limits _{i =1}^{f-1} a_i p ^i$ and $b = \sum\limits_{i = 0} ^ {f-1} b_i p ^i$ with $\{a_i , b_i\} = \{0, 1-\k\}$).

Consequently, $\overline W_\ell |_{I_v} \simeq \omega ^{2h}_{2f} \oplus \omega ^{h (1+\ell)}_{2f} \oplus \omega_{2f} ^{2h\ell} $ or $(\overline W_\ell |_{I_v})^{\t{ss}} \simeq \omega^{2a}_f  \oplus \omega ^{a+b}_{f} \oplus \omega_f ^{2b} $. Since $K(\zeta_\ell)$ is totally ramified at $v$, so is $M$. Denote by $v'$ the only place of $M$ above $v$. Then the inertia subgroup $I_{v'}$ of $\gal (\overline K_v /M_{v'})$ is an index-$3$  subgroup of $I_v$.  Let $\tau$ be a generator of $\gal (M / K)$. We have $\overline W_\ell |_{G_M}\simeq \chi \oplus \chi ^\tau \oplus \chi ^{\tau ^2}$, and similarly for $\overline W_\ell |_{I_{v'}}$.

Therefore the set \{$\chi$, $\chi ^\tau$, $\chi ^{\tau ^2}$\} must match with either  \{$\omega ^{2h}_{2f} ,  \omega ^{h (1+\ell )}_{2f} ,  \omega_{2f} ^{2h\ell} $\} or \{$  \omega^{2a}_f , \ \omega ^{a+b}_{f} , \  \omega_f ^{2b} $\} when restricted to $I_{v'}$. Since the fundamental characters all factor through the tame inertia subgroup, which is commutative,  it is easy to check that the restrictions of $\chi$,  $\chi ^\tau$ and $\chi ^{\tau ^2}$ to $I_{v'}$ are isomorphic and hence identical.  To derive a contradiction, it suffices to show that  $\omega ^{2h}_{2f} ,  \omega ^{h (1+\ell )}_{2f} ,  \omega_{2f} ^{2h\ell } $ restricted to $I_{v'}$ are distinct, and so are the restrictions of $  \omega^{2a} , \ \omega ^{a+b }_{f} , \  \omega_f ^{2b} $ to $I_{v'}$.  By \eqref{eq:fund-cha}, the image of $\omega_{2f}$ is a cyclic group of order $\ell ^{2f}-1$. Since $I_{v'}$ is a subgroup of $I_v$ with index $3$, so  $\omega _{2f}( I_{v'})$ is a cyclic group with order at least $(\ell ^{2f}-1) /3 $. Since $\ell -1 >6\k$,  we see that  $6h , \  3h(1+ \ell), \ 6h\ell$ are distinct inside $\Z/ (\ell ^{2f}-1)\Z$. Then  $\omega ^{2h}_{2f} ,  \omega ^{h (1+\ell )}_{2f} ,  \omega_{2f} ^{2h\ell} $ restricted to $I_{v'}$ are all distinct. The same conclusion can be drawn for the triple $  \omega^{2a}_f , \ \omega ^{a+b}_{f} , \  \omega_f ^{2b} $.
\end{proof}

With the above preparation, we are ready to prove Theorem  \ref{thm-try}.

\begin{proof}[Proof of Theorem \ref{thm-try}]
Let $\ell \in \CL$.~We distinguish two cases.
\begin{itemize}
\item[(I).]  The representation $\eta_\ell$ is strongly irreducible.  By Proposition \ref{prop-residue-irreducible}, $\t{Sym} ^2 \bar \eta_\ell |_{G_{K(\zeta_\ell)}}$ is irreducible, and hence so is  $\sym^2 \bar \eta_\ell|_{G_{F(\zeta_\ell)}}$. Then Theorem \ref{thm-try} follows from Theorem \ref{thm-known}.
\item[(II).] The representation $\eta$ is potentially reducible. Then $\eta_\ell$ is totally odd and automorphic by Proposition \ref{prop-auto}.
\end{itemize}
\end{proof}

\subsection{Applications of Theorem \ref{thm-try}} More about the representations $\eta_\ell$ in Theorem \ref{thm-try} can be concluded provided that we have more information on the characteristic polynomials at the Frobenius elements in $G_K$, as shown below.

\begin{prop}\label{prop-automorphy} Suppose that the representations $\eta _\ell$ in Theorem \ref{thm-try} satisfy the additional condition
\begin{itemize}
\item [(C)] There is a finite set of primes $S$ of $K$ so that at each prime $\gp$ of $K$ outside $S$, the characteristic polynomial of {$\eta_\ell |_{G_K}(\Frob_\gp)$} is independent of the primes $\ell$ not divisible by $\gp$.
\end{itemize}
Then $\{\eta_\ell|_{G_K}\}$ forms a compatible system. If we further assume that $K/F$ is a solvable extension, then $\eta_\ell$ is potentially automorphic for all $\ell$.
\end{prop}

\begin{proof} To prove that $\{\eta_\ell|_{G_K}\}$ forms a compatible system, it suffices to show the existence of a number field  containing coefficients of the characteristic polynomials of $\eta_\ell |_{G_K}(\Frob_\gp)$ for all primes $\gp$ of $K$ not in $S$.

We first assume that there exists a prime $\ell'$  so that $\eta_{\ell'}$ is potentially reducible. Then  Proposition \ref{prop-auto} implies that $\eta_{\ell'}$ is automorphic. In particular, there exists a compatible system of \emph{automorphic} $\ell$-adic Galois representations   $\{ E , S_1, \{ Q_{\tilde \gp} (X)\} , \{\tilde \eta_\lambda \} \} $ of $G_F$ so that $\eta_{\ell '}\simeq \tilde \eta_{\lambda}$ for a prime $\lambda$ of $E$ above $\ell'$. After restricting the representations to $G_K$, we obtain a compatible system  $\{ E , S_2, \{ Q_\gp (X)\} , \{\tilde \eta_\lambda|_{G_K} \} \} $ of representations of $G_K$. Here $S_1$  (resp. $S_2$) is a finite set of places of $F$ (resp. $K$), and $\tilde \gp$ (resp. $\gp$) runs through all finite places of $F$ (resp. $K$) outside $S_1$ (resp. $S_2$), and $E$  contains the coefficients of the characteristic polynomials $Q_{\tilde \gp}(X)$ and  $Q_\gp(X)$. In view of condition (C), for almost all primes $\gp$, $Q_\gp(X)$ is the characteristic polynomial of $\eta_{\ell} |_{G_K}(\Frob_\gp)$ for $\ell = \ell'$ and hence all primes $\ell$ not divisible by $\gp$. Consequently $\{ E,  S_2, \{ Q_\gp (X)\} , \{ \eta_\lambda |_{G_K} \}\}$ forms a compatible system, where $\eta_\lambda = \eta_\ell$ for all primes $\lambda$ of $E$ dividing $\ell$. If $K/F$ is solvable, then by solvable base change theorem in \cite{AC89}, we see that $\{\tilde \eta _\lambda |_{G_K} \}$ is automorphic, so are $\eta_\ell|_{G_K}$.

Next we assume that  for each prime $\ell$,  $\eta_\ell $ is strongly irreducible.
Let $\CL$ be the set  of rational primes as in Theorem \ref{thm-try}. Choose a prime $\ell' \in \CL$.
Then Proposition \ref{prop-residue-irreducible} and Theorem 5.5.1 in \cite{BGGT} imply the existence of a compatible system of $\ell$-adic Galois representations   $\{ E , \tilde S, \{ Q_{\tilde {\gp}} (X)\} , \{\tilde \eta_\lambda \} \}$ of $G_F$ so that $\eta_{\ell'}  \simeq \tilde \eta_{\lambda}$ for a prime $\lambda$ of $E$ above $\ell'$. So by restricting $\{\tilde \eta_\lambda\}$ to $G_K$, the same argument as the potentially reducible case above shows that $\{\eta_\ell|_{G_K}\}$ forms a compatible system. Furthermore, by Theorem \ref{thm-try} there exists a  finite extension $L / F$ so that $ \eta_{\ell'}|_{G_L}$ is automorphic. If $K$ is solvable over $F$, then so is $KL/L$. Hence by solvable base change, $\eta_{\ell'} |_{G_{LK}}$  is automorphic, corresponding to an automorphic representation $\pi$ of $\GL_2$ over $KL$. Now using the facts that the system $\{ \eta_\ell |_{G_{KL}} \}$ is compatible, all $\eta_\ell |_{G_{KL}}$ are irreducible and they are determined by the traces of the elements in $G_{KL}$, we conclude that every $\eta_\ell |_{G_{KL}}$ is automorphic, corresponding to the same representation $\pi$ as  $\eta_{\ell'} |_{G_{KL}}$. This proves that $\eta_\ell$ is potentially automorphic for all primes $\ell$.
\end{proof}

Next we draw some consequences on (potential) automorphy of Scholl representations from Theorem \ref{thm-try}.

\begin{prop}\label{prop-save}{Let $\{\rho_\ell \}$ be a  compatible system of Scholl representations of $G_\Q$.} { Let $K$ be a finite solvable extension of a totally real field $F$.} Suppose that for each $\ell$ we have
\begin{enumerate}
\item $\rho_\ell|_{G_K} \simeq \bigoplus _{i =1} ^d \sigma_{\ell, i }$ where {$\sigma_{\ell, i}$ are degree-$2$ representations of $G_K$;}
\item For each $1 \le i \le d$ there is a {$2$-dimensional representation $\eta_{\ell, i}$ of $G_F$} such that $\eta_{\ell, i}|_{G_K}$ is finitely  projectively equivalent to  $\sigma_{\ell , i}$.
\end{enumerate}
Then $\rho_\ell$ is potentially automorphic for all $\ell$.
\end{prop}

Notice that the representations $\eta_{\ell, i}$ in (2) above are by no means unique. For instance, one may twist $\eta_{\ell, i}$ by finite characters of  $G_F$. Indeed, the following Lemma assures that by so doing we may assume that the  representations $\eta_{\ell, i}$ in (2) are crystalline at each prime $v$ of $F$ above $\ell$ as long as $\rho_\ell |_{G_{\Q_\ell}}$ is crystalline.

\begin{lemma}\label{lem-fix} For each $\eta_{\ell, i}$ in Proposition \ref{prop-save} with $\ell$ large enough so that $\rho_\ell |_{G_{\Q_\ell}}$ is crystalline,
there exists a finite character $\chi_{\ell,i}$ of $G_F$ so that $\chi_{\ell,i} \otimes \eta_{\ell , i}|_{G_{F_v}}$ is crystalline at each prime $v$ of $F$ above $\ell$.
\end{lemma}

\begin{proof}Denote by $V_1$ the ambient space of $\eta_{\ell, i}$ and $V_2$ its dual. Then $(V_1 \otimes V_2)|_{G_K}\simeq \sigma_{\ell, i} \otimes (\sigma_{\ell, i})^\vee  $ is crystalline at all primes of $K$ above large $\ell$.
Now applying Proposition 3.3.4 in \cite{LY16} to $V_1 \otimes V_2$,  we see the existence of  a character $\chi_{\ell,i}$ of $G_F$ with finite image so that $\chi_{\ell,i}\otimes \eta_{\ell , i }$ is crystalline at all primes $v$ of $F$ above $\ell$.
\end{proof}}

Now we proceed to prove Proposition \ref{prop-save}.

\begin{proof} We follow the same idea as the proof of Theorem \ref{thm-try}. Note first that the set $\CL$ of primes constructed in the previous subsection is independent of the decomposition of $\rho_\ell$. {Choose a large $\ell \in \CL$ so that Lemma \ref{lem-fix} holds. Then we may assume that $\eta_{\ell, i}$ in (2) are crystalline at the primes $v$ of $F$ above $\ell$.} Let $J_\ell$ be the set of $i$'s such that $\eta_{\ell, i }$ is strongly irreducible. Then for $i \in J_\ell$ we know that $\sym^2 \bar \eta_{\ell, i }|_{G_{F(\zeta_\ell)}} $ is irreducible by Proposition \ref{prop-residue-irreducible}. Hence Theorem \ref{thm-known} (2) applies to $\eta_{\ell, i}$ for $i \in J_\ell$ and we obtain a finite totally real Galois extension $F'= F'(\ell)$ over $F$ such that $\eta_{ \ell, i }|_{G_{F'}}$ is automorphic for $ i \in J_\ell$. For $i \not \in J_\ell$, we see that $\eta_{\ell, i }$ is induced from a character of the Galois group of a quadratic CM extension of $F$ by Proposition \ref {prop-auto}. It is clear that $\eta_{\ell, i}|_{G_{F'}}$ is also an induction of a character of the Galois group of a quadratic CM extension of $F'$. Thus $\eta_{\ell, i }|_{G_{F'}}$ is automorphic. Since $\eta_{\ell , i}|_{G_K}$ is finitely  projectively equivalent to $\sigma_{\ell , i}$, there is a character  $\chi_i$ of $G_K$ so that $\eta_{\ell, i}|_{G_K} \simeq \chi_i \otimes \sigma_{\ell, i}$ for each $i$. In particular, there exists a finite abelian extension $K_i/K$ so that
$\eta_{\ell , i}|_{G_{K_i}}\simeq \sigma_{\ell , i}|_{G_{K_i}}. $ Let $K'$ be the composite of all $K_i$ and $F'$. Then we see that $K'/F'$ is a finite solvable extension. Thus we conclude that $\sigma_{\ell , i }|_{G_{K'}} = \eta_{\ell , i}|_{G_{K'}}$ is automorphic.
This proves that $\rho_\ell |_{G_{K'}}$
is automorphic for one and hence all $\ell$ as ${\rho_\ell}$ is a compatible system.
\end{proof}

 An immediate consequence of Proposition \ref{prop-save} and Theorem \ref{thm-try} is

\begin{co}\label{co-save} Let $\{\rho_\ell \}$ be a compatible system of Scholl representations of $G_\Q$. Suppose there exists a totally real field $F$ so that for each $\ell$ we have  $\rho_\ell|_{G_F} \simeq \bigoplus _{i =1} ^d \eta_{\ell, i}$ with $ \eta_{\ell, i}$ degree-$2$ representations of $G_F$. Then $\rho_\ell$ is potentially automorphic for all $\ell$. Moreover, if $F = \Q$, then all $\rho_\ell$ are automorphic.
\end{co}

\section{Potentially $\GL_1$ or $\GL_2$-type Scholl representations }\label{sec:GL1&2}
Throughout  this section, $\{ \rho_\ell\}$ denotes a system of compatible $2d$-dimensional Scholl representations of $G_\Q$ associated to a $d$-dimensional space of weight-$\k$ cusp forms of a finite-index subgroup of $\SL$. We assume that \emph{$\rho_\ell$ are   absolutely irreducible  for all $\ell$} in this section.  By  Theorem \ref{thm:Clifford} and  Lemma 5.3.1 of \cite{BGGT} (1), there is a finite Galois extension $L/\Q$ such that, for each $\ell$,
 the restriction $\rho_\ell |_{G_L}$ decomposes as
\begin{equation}\label{eq:splits}
\rho_\ell |_{G_L} \simeq \sigma_{\ell, 1} \oplus  \cdots \oplus \sigma_{\ell, m(\ell)},
 \end{equation}
where $\sigma_{\ell, i}$  are strongly irreducible and conjugate to each other under $\Gal(L/\Q)$, hence they are of the same dimension.

 \begin{remark} \label{rmk-independence} We conjecture that, with $L$ fixed, the number $m(\ell)$ and hence $\dim_{\Qlbar} \sigma_{\ell , i}$ are independent of $\ell$.
\end{remark}

{ In the remainder of this section,  we assume that  $\dim_{\Qlbar} \sigma_{\ell,1}$ is independent of $\ell$, and discuss the case $\dim_{\Qlbar} \sigma_{\ell,1}= 1 $ or $2$.}

\subsection{Scholl representations of $\GL_1$-type}  Consider first the case that $\dim_{\Qlbar} \sigma_{\ell,1}= 1 $.  By Theorem \ref{thm:Clifford},  $\rho_\ell \simeq \Ind ^{G_\Q}_{G_M} (\eta_\ell \otimes \gamma_\ell)$ for some subfield {$M := M(\ell)$} of $L$ depending on $\ell$ a priori. Here $\eta_\ell$ is 1-dimensional and $\gamma _\ell$ is a representation with finite image.

Let $N$ be the splitting field of $\gamma_\ell$. Since $\eta_\ell$ is geometric,  it is  automorphic. Hence $\rho_\ell|_{G_N}$ is automorphic and then $\rho_\ell$ is potentially automorphic. This proves

\begin{theorem}\label{thm:GL_1}Suppose that the Scholl representation $\rho_\ell$ is potentially of $\GL_1$-type for one prime $\ell$. Then $\rho_\ell$ is potentially automorphic for all $\ell$.
\end{theorem}

In general, it is difficult to prove that $\rho_\ell$ is automorphic without further information on $M$ and $\gamma_\ell$.

 \begin{remark} The maximal CM subfield $M_0$ of $M$ can not be totally real. In particular, $M$ can not be $\Q$. Indeed, suppose that $M_0$ is totally real. Then we see from  Lemma \ref{lem:HT} (3) that $\HT_g (\eta_\ell)$ is independent of $g \in \t{Hom}_{\Q} (M , \Qlbar)$. Consequently $ \HT_\tau (\sigma_{\ell, 1}) $  is independent of $\tau \in \t{Hom} _{\Q} (L, \Qlbar)$ (note that $L/\Q$ is Galois). So is $\HT_\tau (\sigma^g_{\ell,1}) $ by Lemma \ref{lem:HT} (2). Then $\rho_\ell|_{G_L} \simeq \sigma_{\ell,1} \oplus \cdots \oplus \sigma_{\ell, m(\ell)}$ has only one Hodge-Tate weight independent of $\tau \in \t{Hom} _{\Q} (L, \Qlbar)$,  contradicting  the fact that $\rho_\ell$ has two distinct Hodge-Tate weights, $0$ and $-\k+1$.
\end{remark}

\subsection{Scholl representations of $\GL_2$-type}

Next we consider the situation that $\dim _{\Qlbar} \sigma_{\ell,1} =2 $ for all $\ell$.  Combining   Theorem \ref{thm:Clifford}  and Proposition \ref{prop-purity}, we have the following result.

\begin{prop} \label{prop:2-splits} Under the notation and assumptions of this section, for each $\ell$ there exist a subfield $M = M(\ell)$ of $L$ and Galois representations $\gamma_\ell : G_M \to \t{GL}_d (\Qlbar)$ and $\eta_\ell : G_M \to \t{GL}_2 (\Qlbar)$ such that
\begin{enumerate}

\item $\gamma_\ell$ has finite image and $\eta_\ell |_{G_L}$ is projectively equivalent to $\sigma_{\ell,1}$;
\item $\rho_\ell \simeq \Ind_{G_M}^{G_\Q} (\gamma_\ell \otimes \eta _\ell)$;
\item  If $M$ is totally real, then for each embedding $\tau: M \inj \overline \Q_{\ell}$ we have  $\t{HT}_\tau (\eta_\ell)= \{0 , 1-\k\}$.

\end{enumerate}
\end{prop}

\begin{remark}\label{rmk-cry-selection} The representations $\gamma_\ell$ and $\eta_\ell$ in the above Corollary are not unique since they may be respectively replaced by $\gamma_\ell \otimes \psi^{-1}$ and $\eta _\ell \otimes \psi$ for any finite character $\psi$ of {$G_M$}. As in the proof of Lemma \ref{lem-fix}, Proposition 3.3.4 in \cite{LY16} implies that,  when $\rho_\ell|_{G_{\Q_\ell}}$ is crystalline, there exists a finite character $\psi$ of $G_M$ so that, after replacing $\eta_\ell$ by $\eta_\ell \otimes \psi$, we have $\eta_\ell |_{G_{M_\q}}$ crystalline at every prime $\q$ of $M$ dividing $\ell$.
\end{remark}

\begin{theorem}\label{thm:main}
Keep the same assumptions and notation as in this subsection.
Assume  further that  there is a finite set of primes $S$ of $L$ so that at each prime $\gp$ of $L$ outside $S$, the characteristic polynomial of $\sigma_{\ell,1}(\Frob_\gp)$ is independent of the primes $\ell$ not divisible by $\gp$.
Then the following statements hold:
\begin{enumerate}
\item The field $M(\ell)$ is independent of $\ell$; denote it by $M$.
\item Assume that $M$ is totally real and $L$ is solvable over $M$, then $\rho_\ell$ is potentially automorphic.\end{enumerate}
\end{theorem}

\begin{proof}(1). Write $\sigma_\ell:= \sigma_{\ell ,1}$,
which is strongly irreducible by assumption. We know from Theorem \ref{thm:Clifford} that the $\sigma_{\ell, i}$'s are conjugates of $\sigma_\ell$ by $G_\Q$. {Recall that an irreducible $2$-dimensional $\ell$-adic Galois representation is determined by its trace},  see  \cite{Serre}.  Hence for $g \in G_\Q$, $\sigma_{ \ell} \simeq \sigma_{\ell}^g$ is equivalent to  $\t{Tr} (\sigma_\ell) = \t{Tr} (\sigma_\ell ^g)$. Let $\ell'$ be another prime.  The assumption that for almost all primes $\gp$ of $L$ the characteristic polynomial of $\sigma _\ell(\Frob_\gp)$  is independent of $\ell$ not divisible by $\gp$  implies that $\t{Tr}( \sigma^g_\ell) = \t{Tr}(\sigma^g_{\ell'})$ for all $g \in G_\Q$ by Cebotarev density theorem. So $\sigma _\ell \simeq \sigma ^g _{\ell}$ if and only if $\sigma _{\ell '} \simeq \sigma ^g _{\ell '}$. Hence $M(\ell)$ is independent of $\ell$.

(2). Since  $\eta_\ell |_{G_L}$ is projectively equivalent to $\sigma_{\ell,1}$ by Proposition \ref{prop:2-splits},  we easily see that (2) is a consequence of Proposition \ref{prop-save}.
\end{proof}

\begin{remark}\label{rmk-compatible}With the assumption of the above Theorem and further assumption that $M$ is totally real,  Proposition \ref{prop-automorphy} shows that $\{ \sigma_{\ell}\}$ also forms a compatible system.
\end{remark}

\subsection{Potentially 2-isotypic case}  More can be said about the automorphy of Scholl representations $\rho_\ell$ when they are potentially 2-isotypic. This is stated in the theorem below.

\begin{theorem}\label{thm:main2}  Let $\{\rho_\ell\}$ be a compatible system of  $2d$-dimensional semisimple subrepresentations of Scholl representations of $G_\Q$ which are all 2-isotypic when restricted to $G_F$ for a finite Galois extension $F/\Q$.
Suppose that $F$ contains a solvable extension $K/\Q$ such that  for  each $\ell$ the representation $\rho_\ell \simeq \Ind_{G_K}^{G_\Q} \sigma_\ell$ for a $2$-dimensional  representation $\sigma_\ell$ of $G_K$.
Then  all $\rho_\ell$ are automorphic.
\end{theorem}

\begin{proof} Since $\{\rho_\ell\}$ forms a compatible system, it suffices to show that $\rho_\ell$ is automorphic for one $\ell$.

As $\rho_\ell$ is potentially 2-isotypic over $F$ and $\rho_\ell \simeq \Ind_{G_K}^{G_\Q} \sigma_\ell$ with  $K \subset F$, this forces $\sigma_\ell$ to be irreducible. Let $\rho'_\ell$ be an irreducible subrepresentation of $\rho_\ell$ so that $\sigma_\ell$ is a subrepresentation of $\rho'_\ell|_{G_K}$. Since $\rho_\ell$ is potentially 2-isotypic,  so is $\rho'_\ell$.  By Theorem \ref{thm:Clifford}, $\rho'_\ell\cong \eta_\ell\otimes \gamma_\ell$ for a $2$-dimensional representation $\eta_\ell$ and a  representation $\gamma_\ell$ of $G_\Q$, where $\eta_\ell |_{G_K}$ is finitely projectively equivalent to $\sigma_{\ell}$ and  $\gamma_\ell$ has finite image.  It follows from Proposition \ref{prop-purity} that $\eta_\ell$ is  irreducible with two distinct Hodge-Tate weights $0$ and $1- \kappa$ for some integer $\kappa \ge 2$.
If $\eta_\ell$ is potentially reducible, then it is odd and automorphic by Proposition \ref{prop-auto}.
So we now assume that $\eta_\ell$ is strongly irreducible for all $\ell$ and show that the same automorphy conclusion holds for some $\eta _\ell$.

By Lemma \ref{lem-fix}, for $\ell$ large so that $\rho_\ell |_{G_{\Q_\ell}}$ is crystalline, there exists a finite character $\xi_\ell$ of $G_\Q$ such that  $\eta_\ell \otimes \xi_\ell$ is crystalline at $\ell$. Replacing $\eta_\ell$ by $\eta_\ell \otimes \xi_\ell$ and $\gamma_\ell$ by $\xi_\ell^{-1}\otimes \gamma_\ell$ if necessary, we may assume that $\eta_\ell$ is crystalline above $\ell$.  Choose a large $\ell$ so that it satisfies the conditions for $\CL$ in \S \ref{ss:3.3} with $F = \Q$ and $K$ as in Theorem \ref{thm:main2}.
As $\eta_\ell |_{G_K}$ is strongly irreducible, by Proposition \ref{prop-residue-irreducible} $\t{Sym}^2 \bar \eta_\ell |_{G_{K(\zeta_\ell)}}$ is absolutely irreducible, and hence the same holds for $\t{Sym}^2 \bar \eta_\ell |_{G_{\Q(\zeta_\ell)}}$. Now apply Theorem \ref{thm-known}(3) to conclude that $\eta_\ell$ is automorphic.

Finally since $K/\Q$ is a finite solvable extension, by base change, $\eta_\ell |_{G_K}$ is also automorphic, and so is its finite twist $\sigma_{\ell}$. Finally, by applying automorphic induction to $\sigma_{\ell}$ over the solvable extension $K/\Q$, we obtain the automorphy of $\rho_\ell \cong \text{Ind}_{G_K}^{G_\Q} \sigma_{\ell}$, as desired.

\end{proof}

\begin{remark}
In \cite{ALLL} the authors showed that degree-$4$ Scholl representations of $G_\Q$ admitting quaternion multiplication are automorphic. As shown in Theorem 3.1.2 {\it loc. cit.}, such representations are special cases of the degree-$4$ representations in Theorem \ref{thm:main2}.
\end{remark}
\bigbreak

\subsection{Some examples of potentially 2-isotypic representations}

\subsubsection{Weight-2 examples}\label{ss:wt-2}

Using Belyi's Theorem, every smooth irreducible projective curve $C$ defined over $\overline \Q$ is isomorphic to the modular curve of a finite index subgroup $\G$ of $\SL$, which is not unique in general. In this regard, the Galois representations on the Jacobian of  $C$, when $C$ is defined over $\Q$, may be viewed as Scholl representations of $G_\Q$ associated to $S_2(\G)$, the space of weight 2 cusp forms of $\G$.
\begin{eg}
Consider the family of curves $C_b: y^2=x^6+bx^3+1$ with generic genus 2. Here we assume $b\in \Q$. On $C_b$ there are two maps, which are $\tau_1: (x,y)\mapsto (\zeta_3 x,y)$ defined over $\Q(\zeta_3)$, and $\tau_2: (x,y) \mapsto (\frac 1x, \frac y{x^3})$ defined over $\Q$. They are of order $3$ and $2$, respectively. Together they generate a finite group isomorphic to the dihedral group of order $12$. For special choices of $b$ such as $b=\pm 2$, the curve has smaller genus and for other choices of $b$ such as $b=0$, the curve is a quotient of a Fermat curve and hence the corresponding 4-dimensional Galois representations $\rho_{\ell,b}$ decompose into $1$-dimensional representations after suitable restriction. For generic $b$, using $\tau_2$ we decompose $\rho_{\ell,b} = \sigma_{\ell,b,1}\oplus \sigma_{\ell, b,2}$ into the sum of two degree-$2$ irreducible representations $\sigma_{\ell, b,i}$, $i=1, 2$,  of $G_\Q$ (since $\tau_2$ is defined over $\Q$) over $\Q_\ell$. Further, for each $i$, as $\ell$ varies, the family $\{\sigma_{\ell, b,i}\}$ is compatible. So the characteristic polynomials of $ \sigma_{\ell,b,i}(\Frob_p)$ at unramified $p$ have coefficients in $\Z$.  Moreover, $\rho_{\ell, b}$ restricted to $G_{\Q(\sqrt{-3})}$ commutes with the operator arising from $\tau_1$.  Hence at almost all primes $p \equiv 1 \mod 3$ splitting in $\Q(\sqrt{-3})$ where $\rho_{\ell, b}$ is unramified, the characteristic polynomial of $\rho_{\ell, b}(\Frob_p)$ is a square. This in turn implies that $\sigma_{\ell, b, 1}(\Frob_p)$ and $\sigma_{\ell, b,2}(\Frob_p)$ have the same characteristic polynomials because they are over $\Z$. When $p \equiv 2 \mod 3$, we have $\Z/p\Z = (\Z/p\Z)^3$ so that over $\Z/p\Z$, the curve $C_b$ is isomorphic to the genus $0$ curve $y^2 = s^2 + bs + 1$. Therefore at almost all such primes $p$, $\rho_{\ell, b}(\Frob_p)$ has trace $0$, which means that $\sigma_{\ell, b, 1}(\Frob_p)$ and $\sigma_{\ell, b,2}(\Frob_p)$ have opposite traces and the same determinants. Combined, this shows that $\sigma_{\ell,b,1}$ and $\sigma_{\ell,b,2}$ differ by a quadratic twist associated to $\Q\left (\sqrt{-3}\right )$. Therefore $\rho_{\ell,b}$ is also potentially 2-isotypic over $\Q\left (\sqrt{-3}\right )$. By Corollary \ref{co-save}, $\rho_{\ell, b}$ is automorphic for all $\ell$.

Both $\sigma_{\ell,b,1}$ and $\sigma_{\ell,b,2}$ are odd and have Hodge-Tate weights $0$ and $-1$, by now established Serre's conjecture, both correspond to holomorphic weight-2 cuspidal newforms with coefficients in $\Z$. As such, the modularity theorem further implies that these newforms are associated to elliptic curves over $\Q$ with conductor equal to the level of the corresponding form. A well-known bound on the conductor of an elliptic curve over $\Q$ then bounds the possible levels, namely the $p$-exponent of the level is at most 8 for $p$ = 2, 5 for $p$ = 3, and 2 for $p \ge 5$. For example, for $b = 1$, the representation $\rho_{\ell, 1}$ and hence $\sigma_{\ell, 1,1}$ and $\sigma_{\ell, 1,2}$ are unramified outside $2, 3, \ell$. Thus the level of the newform $f$ corresponding to the family $\{\sigma_{\ell, 1,1}\}$ divides $2^8\cdot 3^5$. Using the information on the characteristic polynomials of $\rho_{\ell,1}$ at primes $p = 5, 7, 11, 13$ computed by {\texttt Magma} and a search among all weight-2 Hecke eigenforms with levels dividing $2^8\cdot 3^5$ and trivial character,  we conclude that $f$ is the weight-2 level 324 non-CM newform labelled by  324.2.1.a in  \cite{lmfdb}, and that corresponding to $\sigma_{\ell, 1,2}$ is the twist of $f$ by the quadratic character associated to $\Q(\sqrt{-3})$. 

\end{eg}

\begin{eg}\label{eq:8-dimensional}In \cite{DFLST}, the authors used hypergeometric functions over finite fields to study Galois representations arising from hypergeometric abelian varieties. In particular, they considered the  family of  smooth curves $X_t^{[6;4,3,1]}$ obtained from desingularization of the generalized Legendre curves $$y^{6}=x^4(1-x)^3(1-tx)$$  with $t\in \Q\setminus \{0,1\}$. The genus of $X_t^{[6;4,3,1]}$  is 3. It is shown that the Jacobian variety of  $X_t^{[6;4,3,1]}$ has a 2-dimensional  primitive part  $J_t^{\text{prim}}$ defined over $\Q$, obtained from removing the sub abelian varieties  isogenous to factors of the Jacobian varieties obtained from $y^d=x^4(1-x)^3(1-tx)$ where $1<d<6, d\mid 6$. The abelian variety $J_t^{\text{prim}}$ gives rise to a compatible system of $\ell$-adic representations $\rho_{\ell,t}: G_\Q\rightarrow \text{GL}_4(\Q_\ell)$.  Due to the map $(x,y)\mapsto (x,\zeta_{6} y)$ on $X_t^{[6;4,3,1]}$ defined over $K=\Q(\zeta_3)$, $$\rho_{\ell,t}|_{G_K}=\sigma_{\ell,t}\oplus \sigma_{\ell,t}^{\tau}$$ where $\tau$ is the complex conjugation in $\gal(K/\Q)$. By Example 3 of \cite{DFLST}, there is a finite character $\psi$ of $G_K$ trivial on $G_{L_t}$ such that $\sigma_{\ell,t}\cong \sigma_{\ell,t}^{\tau}\otimes \psi$. Here $L_t=K\left (\sqrt[6]{t\frac{(1-t)^2}{2^4}} \right)$ is a finite Galois extension of  $\Q$. Thus $\rho_{\ell, t}$ is $2$-isotypic over $L_t$. For any $t\in \Q\smallsetminus \{0,1\}$ such that $L_t\neq K$ and $\sigma_{\ell,t}$ is strongly irreducible, $\sigma_{\ell,t}$ is not isomorphic to  $\sigma_{\ell,t}^{\tau}$.  Hence $\rho_{\ell, t}=\text{Ind}_{G_K}^{G_\Q}\sigma_{\ell,t}.$ For those values of $t$, we have $\rho_{\ell, t}$ automorphic for all $\ell$ by Theorem  \ref{thm:main2}.

Similarly, when one considers the smooth model  $X_t^{[12;9,5,1]}$ of $y^{12}=x^9(1-x)^5(1-tx)$, for generic $t\in \Q\setminus\{0,1\}$, the primitive part of its Jacobian variety is 4-dimensional and the corresponding 8-dimensional Galois representation $\rho_{\ell, t}$ is potentially of $\GL_2$-type such that its restriction to $G_K$ with $K=\Q(\zeta_{12})$ decomposes as
$$\rho_{\ell,t}|_{G_K}=\sigma_{\ell,t} \oplus \sigma_{\ell,t}^{\tau_1}\oplus  \sigma_{\ell,t}^{\tau_2}\oplus  \sigma_{\ell,t}^{\tau_1\tau_2},$$  where $\tau_1:\zeta_{12}\mapsto \zeta_{12}^{-1}$ and $\tau_1:\zeta_{12}\mapsto \zeta_{12}^{5}$  are in $\gal(K/\Q)$ and $\sigma_{\ell, t}$ is strongly irreducible for generic $t$. By  \S 7.2 of \cite{DFLST}, $\sigma_{\ell,t} \cong \sigma_{\ell,t}^{\tau_1}\otimes \chi_1(t)$ and $\sigma_{\ell,t} \cong \sigma_{\ell,t}^{\tau_2}\otimes \chi_2(t)$, where $\chi_1(t),\chi_2(t)$ are characters of $G_K$ of order dividing $12$ whose kernels are the absolute Galois group of $K \left (\sqrt[12]{-27t^2(1-t)^6}\right )$ and $K(\sqrt[6]{t})$ respectively. Thus, for a generic choice of $t$,  none of $\chi_1(t)$, $\chi_2(t)$, $\chi_1(t)\chi_2(t)$  are trivial characters, indicating that the representations $\sigma_{\ell,t},  \sigma_{\ell,t}^{\tau_1}, \sigma_{\ell,t}^{\tau_2}, \sigma_{\ell,t}^{\tau_1\tau_2}$ are pairwise  non-isomorphic.  Hence $\rho_{\ell, t}=\text{Ind}_{G_K}^{G_\Q}\sigma_{\ell,t}.$ On the other hand, $\rho_{\ell,t}$ is potentially 2-isotypic over the Galois extension $L_t=\Q\left (\zeta_{12}, \sqrt[6]{t},\sqrt[12]{-27(1-t)^6}\right )$ of $\Q$. Therefore, by Theorem \ref{thm:main2}, $\rho_{\ell, t}$ is automorphic for a generic $t$.
\end{eg}

\subsubsection{A weight-4 example}The following example was originally investigated by the late Oliver Atkin among his unpublished notes. The construction is similar to the cases discussed in \cite{LLY, HLV} except that the space of modular forms in consideration has weight-4 instead of weight-3.
 The group $\G_0(4)$ has genus $0$, 3 cusps, and admitting the Atkin-Lehner involution $W_4$. A Hauptmodul for $\G_0(4)$ is
$t:=t(z)=\eta(z)^8/\eta(4z)^8$, where $\displaystyle \eta(z):=q^{1/24}\prod_{n=1}^\infty (1-q^n), q=e^{2\pi i z}$ denotes the Dedekind eta function. The function $t$ has a simple pole at the cusp infinity and vanishes at the cusp 0.  The function
$E:=\eta(2z)^{16}/\eta(z)^8 $ is a  weight 4 Eisenstein series for
$\G_0(4)$ which vanishes at all cusps of $\G_0(4)$ but 0. The
function $t_3(z):=\sqrt[3]{\eta(z)^8/\eta(4z)^8} = \sqrt[3]{t}$ is a Hauptmodul of an
index-3 noncongruence subgroup, denoted by $\G$, of $\G_0(4)$.  One way to see that $t_3(z)$ is a noncongruence modular function is that  the Fourier coefficients of its $q$-expansion have unbounded powers of $3$ in the denominators.
The modular curve for $\G$ has a model  defined over $\Q$ and it is a 3-fold cover of the modular curve for
$\G_0(4)$, ramified only at the cusps $0$ and $\infty$ with
ramification degree 3. The space $S_4(\G)$ of weight 4 cusp forms  is
spanned by
$$f_1=E\cdot (t(z))^{2/3}=\sqrt[3]{\frac{\eta(2z)^{48}}{ \eta(z)^8\eta(4z)^{16}}} \quad {\rm ~~and} \quad
\quad f_2=E\cdot (t(z))^{1/3}=\sqrt[3]{\frac{\eta(2z)^{48}}{
\eta(z)^{16}\eta(4z)^{8}}}.$$ Their Fourier expansions $\displaystyle f_i=\sum_{n\ge 1} a_i(n)q^{n/3}$,
$i = 1, 2$, have coefficients in $\Q$.

 Let $\rho_\ell$ be the $\ell$-adic $4$-dimensional Scholl representation of $G_\Q$ attached to $S_4(\G)$.  The matrices $\zeta=\begin{pmatrix}
  1&1\\0&1
\end{pmatrix}$ and  $W_4=\begin{pmatrix}
0&-1\\4&0
\end{pmatrix}$ normalize $\Gamma$ and thus act on forms on $\Gamma$ via the stroke operator as follows:
$$ f_1|_\zeta=\zeta_3^{~2}f_1, \quad  f_2|_\zeta=\zeta_3f_2, \quad f_1|_{W_4}=2^{16/3}f_2, \quad  f_2|_{W_4}=2^{8/3}f_1.$$
They induce the corresponding actions $\zeta^*$ and $W_4^*$ on  $\rho_{\ell}$.  The operator $\zeta^*$ is defined over $K=\Q\left (\sqrt{-3} \right)$ and $W_4^*$ is defined over $F = K\left (\sqrt[3]{2} \right)$, the splitting field  of $x^3-2$.  These two operators generate a group isomorphic to $S_3$, the symmetric group on $3$ letters, which acts on $\rho_\ell$ and commutes with $\rho_\ell(G_F)$.

Using $\zeta^*$, we decompose $\rho_\ell|_{G_{K}}=\sigma_\ell\oplus \sigma_\ell^\tau$, where  $\sigma_\ell$ is a 2-dimensional strongly irreducible representation of $G_K$ over $\Q(\zeta_3)$
 and $\tau\in G_\Q \smallsetminus G_K$. Explicit computations yield the following characteristic
polynomials of Frobenius elements and their factorization over $\Z[\zeta_3]$. More precisely, we proceed as follows. We first compute many Fourier coefficients $a_i(n)$ of $f_i$ explicitly. As shown by Scholl \cite{sch85b},  for $p \ge 5$, the characteristic polynomial of $\rho_\ell(\Frob_p)$ is a degree-$4$ polynomial $T^4+A_3(p)T^3+A_2(p)T^2+A_1(p)T+A_0(p) \in \Z[T]$ with all roots of absolute value $p^{3/2}$. This sets the range for $A_j(p)$. Moreover, the Atkin and Swinnerton-Dyer congruences hold for $f_i$, $i = 1, 2$, and integers $r \ge -1$: $$a_i(p^{r+2})+A_3(p)a_i(p^{r+1})+A_2(p)a_i(p^{r})+A_1(p)a_i(p^{r-1})+A_0(p)a_i(p^{r-2})\equiv0 \mod p^{3+r}.$$ Here $a_i(p^{r})=0$ if $r<0$.
See \cite{sch85b} for details. For each prime $p$ listed below, we tested for the first few $r\ge 1$ and $i=1,2$, from which all coefficients $A_j(p)$ are determined.
$$\begin{tabular}{|c|c|c|c}\hline
$p$&Char. poly. of $\rho_\ell(\Frob_p)$ & Factorization over $\Z[\zeta_3]$\\
\hline
5&$x^4-74 x^2+5^6$&$(x^2+18 x+5^3) (x^2-18 x+5^3)$\\
\hline 7&$x^4+8x^3-279x^2+2744x+7^6$&$(x^2-8\zeta_3x+\zeta_3^2 7^3)(x^2-8\zeta_3^2x+\zeta_3 7^3)$\\
\hline 11&$x^4+1366 x^2+11^6$&$(x^2+36 x+11^3) (x^2-36 x+11^3)$\\
\hline 13& $x^4-10x^3-2097x^2-21970x+13^6$&$(x^2+10\zeta_3x+\zeta_3^2 13^3)(x^2+10\zeta_3^2x+\zeta_3 13^3)$\\
\hline 17&$x^4+9502 x^2+17^6$&$(x^2-18 x+17^3) (x^2+18 x+17^3)$ \\
\hline 23&$x^4+19150 x^2+23^6$&$(x^2+72 x+23^3) (x^2-72
x+23^3)$\\
\hline 31&$(x^2+16x+31^3)^2$&$(x^2+16x+31^3)^2$. \\ \hline
\end{tabular}$$From the data above we see that $\sigma_\ell$ is not isomorphic to  $\sigma_\ell^\tau$; thus $\rho_\ell\cong \text{Ind}_{G_K}^{G_\Q}\sigma_\ell$. Moreover, $\sigma_\ell|_{G_F}\cong \sigma_\ell^\tau|_{G_F}$  by the action of $W_4^*$. Hence $\rho_{\ell}$ is 2-isotypic over the Galois extension $F=\Q(\sqrt{-3},\sqrt[3]{2})$ of $\Q$. By Theorem \ref{thm:main2}, $\rho_\ell$ is automorphic. Alternatively, applying Theorem \ref{thm:Clifford}, we have $\rho_\ell \cong \eta_\ell \otimes \gamma_\ell$, where $\eta_\ell|_{G_K}$ is finitely projectively equivalent to $\sigma_\ell$ and $\gamma_\ell$ has finite image. In fact, according to the numerical data computed by Atkin,  $\eta_\ell$ can be chosen to be the Galois representation attached to either one of the following weight-4 level 12 congruence Hecke eigenforms $$g_\pm:=g_1(z)\pm18g_5(z)+3(g_1(3z)\pm18g_5(3z))$$ with $g_1(z)=\eta^4(z)\cdot (3E_2(3z)-E_2(z))/2$ and $ g_5=\eta(z)^2 \eta(3z)^6$ in which $E_2$ is the
weight 2 nonholomorphic Eisenstein series. Here $g_\pm$ differ by twisting by the quadratic character $\chi_{-3}$ corresponding to $\Q(\sqrt{-3})$. The representation $\gamma_\ell$ is induced from a finite character of $G_K$, hence is also automorphic. Therefore $\rho_\ell$ corresponds to an automorphic representation of $\GL_2 \times \GL_2$ over $\Q$, hence is automorphic by \cite{Ramakrishnan}.

\section{An infinite family of potentially automorphic Scholl representations attached to weight-$3$ noncongruence forms}\label{sec:wt-3}

In this section we apply the results of previous sections to prove the potential automorphy of an infinite family of Scholl representations attached to weight-$3$ cusp forms for the noncongruence subgroups explicitly constructed in \cite{ALL}. This gives the first family of Scholl representations of $G_\Q$ with unbounded degree for which the potential automorphy is established.

\subsection{A family of elliptic surfaces $\E_n$}\label{sec;En} It is well-known that there are thirteen $K3$ surfaces defined over $\Q$ whose N\'eron-Severi group has rank $20$, generated by algebraic cycles over $\Q$. Elkies and Sch\"utt \cite{Elkies-Schutt} have constructed them from suitable double covers of $\mathbb P^2$ branched above $6$ lines. We recall  such a $K3$ surface $\E_2$, labelled as $\mathcal A(2)$ in \cite{Beukers-Stienstra}  by Beukers and Stienstra. It is given by replacing the variable $\tau$ in the equation $X(Y-Z)(Z-X)-\tau (X-Y)YZ=0$ by $t_2^{~2}$, yielding a $2$-fold cover of $\mathbb P^2$ branched over the $6$ lines $X=0, Y=0, Z=0, X-Y=0,  Y-Z=0, Z-X=0$  positioned as follows:
$$
\includegraphics[scale=0.25]{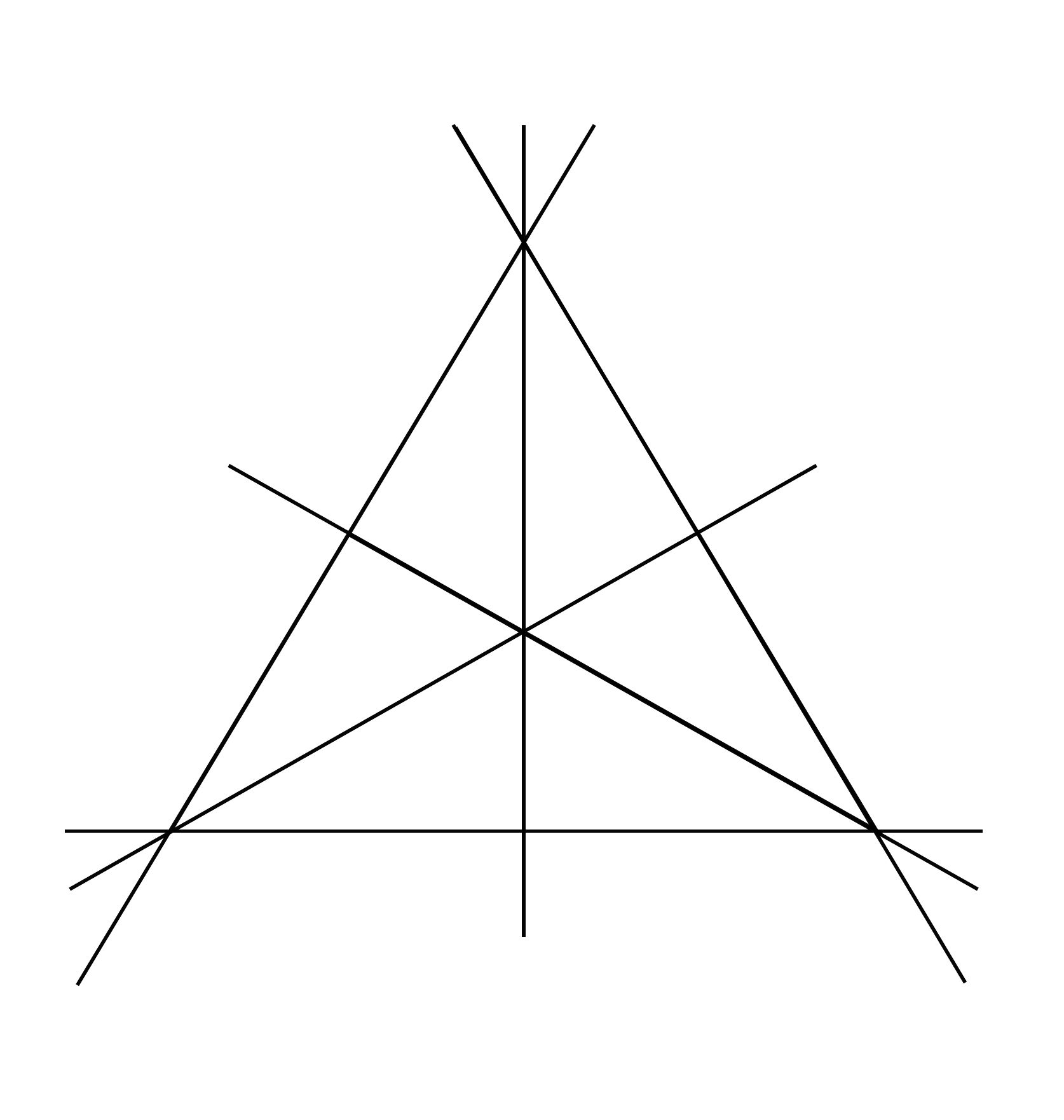}
$$

By letting $X=-t_2^{~2}VW, Y=-t_2^{~2}UW+U^2, Z=-UV$ and further by $x=U/W, y=V/W$,   Beukers and Stienstra \cite{Beukers-Stienstra} arrived at the following nonhomogeneous model for the $K3$ surface $\E_2$ in the sense of Shioda \cite{Shioda} $$ \E_2~: ~y^2+(1-t_2^{~2})xy-t_2^{~2} y=x^3-t_2 ^{~2} x^2,$$ where $t_2$ is a parameter.  We  extend this setting to the family of elliptic surfaces $\E_n$, for $n \ge 2$, in the sense of Shioda  defined by
\begin{eqnarray}\label{eq:En}
\E_n~:~ y^2+(1-t_n^{~n})xy-t_n^{~n} y=x^3-t_n^{~n} x^2
\end{eqnarray}

\noindent with $t_n:=\sqrt[n]{\tau}$ as a parameter. It is an $n$-fold cover of $\mathbb P^2$ branched above the same configuration of $6$ lines. The Hodge diamond of $\E_n$ is of the form

$$\begin{array}{cccccccc}
&&1&&\\
&0&&0&&\\
(n-1)&&10n&&(n-1)\\
&0&&0&&\\
&&1&&\\
\end{array}$$

The action of the Galois group $G_\Q$ on the $(12n-2)$-dimensional space $H^2_{et}(\E_n \otimes_\Q \overline \Q, \Q_\ell)$ is an $\ell$-adic representation $\tilde \rho_{n, \ell}$. As $\ell$ varies, they  form a compatible system such that, at a prime $p$ where $\E_n$ has good reduction, $\tilde \rho_{n, \ell}$ for $\ell \ne p$ is unramified at $p$ and $\tilde \rho_{n, \ell}(\Frob_p)$ has characteristic polynomial $P_2(\E_n,p, T) \in \Z[T]$ independent of $\ell$ and of degree $12n-2$. It occurs in the denominator of the Zeta function of $\E_n$ mod $p$:

$$ Z(\E_n/\F_p, T) = \frac{1}{(1-T)(1-p^2T)P_2(\E_n,p,T)}.$$

\noindent Moreover, the algebraic cycles on $\E_n$ generate the N\'eron-Severi group of rank $10n$; its orthogonal complement in the second singular cohomology group of $\E_n$ is a group of rank $2n-2$ generated by  transcendental cycles. Each group gives rise to a subrepresentation of $\tilde \rho_{n, \ell}$, denoted $\tilde \rho_{n, \ell, a}$ (algebraic part) and $\tilde \rho_{n, \ell, t}$ (transcendental part), respectively, so that $\tilde \rho_{n, \ell} = \tilde \rho_{n, \ell, a} \oplus \tilde \rho_{n, \ell, t}$. Therefore $P_2(\E_n,p,T)$ is a product of  $10n$ linear factors  in $\Z[T]$ from counting points on algebraic cycles and a degree $2n-2$ polynomial $Q(\E_n; p; T) \in \Z[T]$  from counting points on transcendental cycles. The system of $(2n-2)-$dimensional $\ell$-adic representations $\{ \tilde \rho_{n, \ell, t}\}$ of $G_\Q$ is compatible. The factor at a good prime $p$ of the associated $L$-function is $1/Q(\E_n; p, p^{-s})$.

\subsection{Fibration of $\E_n$ over a modular curve $X_n$} It was shown in \cite{Beukers-Stienstra} that the elliptic surface $\E$ defined by
$$ y^2+(1-\tau)xy-\tau y=x^3-\tau x^2$$
with parameter $\tau$ is fibered over the genus $0$ modular curve $X_{\G^1(5)}$ of the congruence subgroup
\begin{eqnarray*}
\Gamma^1(5) &=&\left\{\gamma \in \SL ~|~ \gamma \equiv \left( \begin{matrix}
             1 & 0\\
             * & 1
              \end{matrix}
            \right)
         \mod 5 \right\}
\end{eqnarray*}
of $\SL$.
Thus $\E_n$ is fibered over a genus zero $n$-fold cover $X_n$ of $X_{\Gamma^1(5)}$ under $\tau = t_n^{~n}$. We give more details about $X_n$. The curve $X_{\Gamma^1(5)}$ is defined over $\Q$, containing no elliptic points and $4$ cusps, at $\infty$, $0$, $-2$ and  $-5/2$, among them the cusps $\infty$ and $-2$ are defined over $\Q$.  Let $E_1$ and $E_2$ be two Eisenstein series of weight $3$ with $\Q$-rational Fourier coefficients and having simple zeros at all cusps
except $\infty$ and $-2$, respectively.
Then $\tau = \frac{E_1}{E_2}$ is a Hauptmodul for $\Gamma^1(5)$ with a simple zero at the cusp $-2$ and a simple pole at the cusp $\infty$.
With $t_n = \sqrt[n] {\tau}$, the curve $X_n$ is unramified over $X_{\Gamma^1(5)}$ except totally ramified above the cusps $\infty$ and $-2$ (with $\tau$-coordinates $\infty$ and $0$, resp.). This describes the index-$n$ normal subgroup $\Gamma_n$ of $\Gamma^1(5)$ such that $X_n$ is the modular curve of $\Gamma_n$. See \cite{ALL} for an expression of $\Gamma_n$ in terms of generators and relations. Note that $\E_n$ is an universal elliptic curve over $X_n$.

It is known that $\Gamma_n$ is a noncongruence subgroup of $SL_2(\Z)$ if $n \ne 5$, and $\Gamma_5$ is isomorphic to the principal congruence subgroup $\Gamma(5)$. The space of weight-$3$ cusp forms for $\Gamma_n$ is $(n-1)$-dimensional, corresponding to holomorphic $2$-differentials on $\E_n$; it has an explicit basis given by $(E_1^jE_2^{n-j})^{1/n}$ for $1 \le j \le n-1$ (cf. \cite{ALL, LLY}).

\subsection{Scholl representations attached to $S_3(\Gamma_n)$} As reviewed in \S \ref{sec: Scholl}, to $S_3(\Gamma_n)$ Scholl has attached a compatible system of $2(n-1)$-dimensional $\ell$-adic representations $\rho_{n, \ell}$ of $G_\Q$ acting on the parabolic cohomology group $W_{n, \ell} = H^1(X_n\otimes_\Q \overline {\Q}, \iota_* \mathcal F_\ell)$ of $X_n$,
similar to Deligne's construction of $\ell$-adic Galois representations attached to congruence forms (cf.  \cite{sch85b}). He also proved in \cite{sch90} the existence of a Kuga-Sato variety $Y_n$ over $\Q$ of dimension $2$ such that $W_{n, \ell}$ can be  imbedded into the $G_\Q$-module $H_{et}^2(Y_n\otimes_\Q \overline \Q, \Q_\ell)$.  In our case $Y_n$ is nothing but $\E_n$ and the subrepresentation of $\tilde \rho_{n, \ell}$ isomorphic to $\rho_{n, \ell}$ is precisely $\tilde \rho_{n, \ell, t}$. Hence $L(\{\rho_{n, \ell}\}, s)$ agrees with $L(\{\tilde \rho_{n,\ell,t} \}, s)$.

We list some key properties of the Scholl representations $\rho_{n, \ell}$:
\begin{enumerate}

\item $\rho_{n, \ell}$ is unramified outside $n\ell$;

\item For $\ell$ large, $\rho_{n, \ell}|_{G_{\Q_\ell}}$ is crystalline with Hodge-Tate weights $0$ and $-2$, each with multiplicity $n-1$;

\item $\rho_{n, \ell}$ at the complex conjugation has eigenvalues $\pm1$, each with multiplicity $n-1$;

\item The space $W_{n,\ell}$ of $\rho_{n, \ell}$ admits the action by $\zeta = \left( \begin{matrix}
             1& 5\\
             0 & 1
              \end{matrix}
            \right)$, induced from the action
 \begin{equation}\label{eq:A}
\zeta(t_n) = \zeta_{n}^{-1}t_n\end{equation}
on rational functions on $X_n$. Here $\zeta_n = e^{2 \pi \sqrt{-1}/n}$.
\end{enumerate}

The purpose  of \S5 is to prove the (potential) automorphy of the Scholl representation $\rho_{n,\ell}$. Since $\Gamma_5$ is a congruence subgroup, $\rho_{5, \ell}$ is naturally automorphic. Our concern is for the case $n \ne 5$. To proceed, we make the following observation. Given $n \ge 2$, it follows from the property (4) above that the action of $\zeta^*$ on $W_{n, \ell}$ induced from $\zeta$ has order $n$, so $W_{n, \ell}$ decomposes into the direct sum of eigenspaces of $\zeta^*$ with eigenvalues $\zeta_n^m$ for $m = 1,..., n-1$. For a proper divisor $d$ of $n$ with $d>1$, since $t_d = t_n^{n/d}$, the subspace of $W_{n, \ell}$ on which $(\zeta^*)^{d}$ acts trivially can be identified with $W_{d,\ell}$ so that $\rho_{d, \ell}$ may be regarded as a subrepresentation of $\rho_{n, \ell}$. The space $W_{d,\ell}$ is the sum of eigenspaces with eigenvalues $\zeta_d^r = \zeta_n^{rn/d}$ for $r=1,..., d-1$. Consequently the sum of eigenspaces in $W_{n, \ell}$ with eigenvalues $\zeta_n^m$ where $(m, n) = 1$ complements the space $\sum_{d|n, 1 < d < n} W_{d, \ell}$. This $G_\Q$-invariant subspace is the ``new'' part of $W_{n, \ell}$, denoted by $\rho_{n, \ell}^{new}$. By inclusion-exclusion, we find that $\rho_{n, \ell}^{new}$ has dimension $2\phi(n)$ (where $\phi$ is Euler's totient-function) and $\rho_{n, \ell}$ can be decomposed as the sum

$$ \rho_{n, \ell} = \bigoplus_{d|n, ~d>1} \rho_{d, \ell}^{new}.$$

\noindent Clearly,  $\rho_{n, \ell} = \rho_{n, \ell}^{new}$ when $n$ is a prime. It suffices to show that each $\rho_{n,\ell}^{new}$ is (potentially) automorphic.  For this, we shall prove

\begin{theorem}\label{thm:potentialauto}
 Let $n\ge 2$ be an integer. There are $2$-dimensional $\ell$-adic representations $ \sigma_\ell $ of $G_{\Q(\zeta_n)}$ whose semi-simplifications form a compatible system such that $\rho_{n,\ell}^{new} \cong \text{Ind}_{G_{\Q(\zeta_n)}}^{G_\Q} \sigma_\ell$ for all primes $\ell$. For each $\ell$ the representation $\rho_{n,\ell}^{new}$ is potentially automorphic. Further, it is automorphic if either $n \le 6$ or $\sigma_\ell$ is potentially reducible.
\end{theorem}

We recall the known result in the literature that $\rho_{n, \ell}^{new}$ is automorphic for $n \le 6$ and $n \ne 5$.  First,   two-dimensional Scholl representations of $G_\Q$ are
automorphic by  Theorem \ref{thm-known}.  In particular, $\rho_{2, \ell}$ is modular. In fact Beukers and Stienstra have already shown  in \cite{Beukers-Stienstra} that $\tilde \rho_{2, \ell, t}$ ($\cong \rho_{2, \ell}$) is  isomorphic to the $\ell$-adic Deligne representation attached to the congruence cusp form $\eta(4\tau)^6$. Using the method of Faltings-Serre, Li, Long and Yang proved in \cite{LLY} the automorphy of the $4$-dimensional $\rho_{3, \ell}$, which corresponds to an automorphic representation
of  $\GL_2\oplus \GL_2$ over $\Q$. When $n=4$,
 it was observed in \cite{ALL} that the space of $\rho_{4, \ell}^{new}$ admits quaternion multiplication. Using these symmetries they showed that the representation $\rho_{4, \ell}^{new}$ is automorphic, corresponding to an automorphic representation of $\GL_2 \times \GL_2$ over $\Q$, and hence also an automorphic representation of $\GL_4$ over $\Q$ by a result of Ramakrishnan \cite{Ramakrishnan}. Thus $\rho_{4, \ell}$ is automorphic. The same holds for $\rho_{6, \ell}^{new}$ by a similar argument carried out by Long  \cite{Long08}. Thus $\rho_{6, \ell}$ is automorphic. The paper \cite{ALLL} by Atkin, Li, Liu, and Long gives a conceptual explanation of the automorphy of $4$-dimensional Galois representations with QM,   including the cases $n=3,4,6$.

\subsection{The structure of $\rho_{n,\ell}^{new}$}
 Since the action of $\zeta$ is defined over the cyclotomic field $\Q(\zeta_n)$, each eigenspace of $\zeta$ is $G_{\Q(\zeta_n)}$-invariant. Denote the subrepresentation of $\rho_{n, \ell}|_{G_{\Q(\zeta_n)}}$ on the eigenspace with eigenvalue $\zeta_n^i$ by $\sigma_{n, \ell, i}$ so that

$$\rho_{n, \ell}|_{G_{\Q(\zeta_n)}} = \bigoplus_{1 \le i \le n-1} \sigma_{n, \ell, i}$$ and

$$\rho_{n, \ell}^{new}|_{G_{\Q(\zeta_n)}} = \bigoplus_{1 \le i \le n-1, ~(i, n)=1} \sigma_{n, \ell, i}.$$

As computed in \cite{Beukers-Stienstra}, the model (\ref{eq:En}) for $\E_n$ came from the homogenous model \begin{equation}\label{eq:5.4}X(Y-Z)(Z-X)-t_n^n (X-Y)YZ=0\end{equation} for the surface by setting
$X=-t_n^n y,Y=-t_n^n x-x^2, Z=-xy$. If, instead, we let $x=\frac ZX,y=\frac YZ$ and $s=t_n\cdot \frac{ X^2}{Y(X-Y)}$, then we get the following model
\begin{eqnarray}\label{eq:En2}
\E_n ~:~ s^n = (xy)^{n-1}(1-y)(1-x)(1-xy)^{n-1} =: f_n(x,y)
\end{eqnarray} which is more amenable to our computation. Observe that the action of $\zeta$ on $t_n$ in (\ref{eq:A})  translates to $\zeta(s) = \zeta_n^{-1}s$ for the model (\ref{eq:En2}).

Given $g \in G_\Q$, its action on $\Q(\zeta_n)$ is determined by the image $g (\zeta_n) = \zeta_n ^{\e(g)}$, where the exponent $\e(g)$ lies in $(\Z/ n \Z)^ \times $. For $(x, y, s) \in \E_n (\overline \Q)$ defined by the model (\ref{eq:En2}), we have
$$ g \circ \zeta (x, y, s) = g(x, y, \zeta_n^{-1}s) = (g(x),  g(y), g(\zeta_n)^{-1} g(s)) = \zeta^{\e(g)} \circ g(x, y, s),$$
that is, $g \circ \zeta = \zeta^{\e(g)} \circ g.$ Since the space of $\rho_{n,\ell}$ is contained in $H^2_{et} ({\E_n} \otimes_{\Q} \overline \Q, \Q_\ell )$, this relation between $g$ and $\zeta$ implies that the induced action of $g$ on $\rho_{n,\ell}$ sends the $\zeta_n^i$-eigenspace of $\zeta$ to the $g(\zeta_n)^i$-eigenspace of $\zeta$. Therefore the conjugate of $\sigma_{n, \ell, i}$ by $g$ is isomorphic to $\sigma_{n, \ell, \e(g)i}$. Consequently, for a fixed $n$, the $\sigma_{n, \ell, i}$'s with $(i,n)=1$ are conjugates of $\sigma_{n, \ell, 1}$ by $\gal(\Q(\zeta_n)/\Q)$, hence they are all $2$-dimensional and
 $\rho_{n , \ell} ^{new} = \Ind_{G_{\Q(\zeta_n)}}^{G_\Q} \sigma_{n , \ell, 1}$.
Further, for the complex conjugation $c$ in $G_\Q$, we have $\e(c) = -1$ so that, for $1 \le i \le n-1$,
$$ \sigma_{n, \ell, i}^c = \sigma_{n, \ell, n-i}.$$

We record the above discussion in

\begin{lemma}\label{lem-structrue}
\begin{enumerate}

\item For $g \in G_\Q$ and $1 \le i \le n-1$, we have $\sigma_{n, \ell, i}^g \cong \sigma_{n, \ell, \e(g)i},$ where $\e(g)$ is such that $g(\zeta_n) = \zeta_n ^{\e(g)}$. Therefore all $\sigma_{n, \ell, i}$ are $2$-dimensional, and
 $$\displaystyle \rho_{n , \ell} ^{new}|_{G_{\Q(\zeta_n)}} = \oplus_{g \in \gal(\Q(\zeta_n)/\Q)} \sigma_{n,\ell,i}^g \quad \text{ and } \quad  \rho_{n, \ell}^{new} \cong \Ind_{G_{\Q(\zeta_n)}}^{G_\Q} \sigma_{n , \ell, i} $$ for any $i$ coprime to $n$.

\item For the complex conjugation $c \in G_\Q$, we have $\sigma_{n, \ell, i}^c \cong \sigma_{n, \ell, n-i}.$
\end{enumerate}
\end{lemma}

\subsection{Computing the trace of $\sigma_{n,\ell,i}$}

The aim of this subsection is to express the trace of $\sigma_{n,\ell,i}$ at Frobenius elements in terms of character sums by using the isomorphism $\rho_{n,\ell} \cong \tilde \rho_{n,\ell,t}$. As a result, the semi-simplification of $\sigma_{n, \ell, i}$, as $\ell$ varies, forms a compatible system.

Let $K$ be a finite extension of the cyclotomic field $\Q(\zeta_n)$. At a place $\gp$ of  $K$ not dividing $n$ with residue field $k_{\gp}$ of cardinality $N\gp$, we have $n | N\gp-1$. The  $n$th power residue symbol at $\gp$, denoted  by $\left ( \frac{}{\gp} \right)_n$,  is the $\langle \zeta_n \rangle \cup \{0\}$-valued function defined by
$$ \left ( \frac{a}{\gp} \right)_n \equiv a^{(N\gp-1)/n}  \pmod \gp \quad {\rm for~all~} a \in  \Z_K,$$ where $  \Z_K$ denotes the ring of integers of $K$. It induces a character $\xi_{\gp, n}$ of order $n$ on $k_{\gp}^\times$ extended to $k_\gp$ by $\xi_{\gp,n}(0)=0$. Also for fixed nonzero $a \in \Z_K$, as $\gp$ varies among the finite places of $K$  not dividing $na$, the power residue symbol defines a representation of the Galois group $\text{Gal}(K(\sqrt[n]{a})/K)$ such that
 \begin{equation}\label{eq:msymbol}
\left (\frac{a}{\fp}\right)_n=\frac{\AF_\fp(\sqrt[n]{a})}{\sqrt[n]{a}},
\end{equation} where $\AF_\fp$ is the the arithmetic Frobenius at  $\fp$.  This fact, originally due to Hilbert, is part of the class field theory.  See, for example, \cite{WIN3X}, sections  5 and 6 for more detail.

We analyze the rational points on $\E_n$ using the model (\ref{eq:En2}) given by
\begin{eqnarray*}
\E_n ~:~ s^n = (xy)^{n-1}(1-y)(1-x)(1-xy)^{n-1} =: f_n(x,y).
\end{eqnarray*}
The solutions to the above equation with $s=0$ lie on algebraic cycles. At a place $\gp$ of $K$ not dividing $n$, $\E_n$ has good reduction mod $\gp$ and the number of solutions to (\ref{eq:En2}) mod $\gp$ with $s \ne 0$ can be expressed in terms of the character sum
\begin{eqnarray}\label{ptsonEn}
 \sum_{1 \le i \le n} \sum_{x, y \in k_\gp} \xi_{\gp,n}^i(f_n(x,y)).
\end{eqnarray}
Noticing that $\xi_{\gp, n}$ has order $n$, we can rewrite the inner sum as
$$\sum_{ \overset{x, y \in k_\gp^\times}{ x \ne 1, y \ne 1, xy \ne 1}} \xi_{\gp,n}^i(g(x,y)),$$
where $g(x, y) = \frac{(1-x)(1-y)}{xy(1-xy)}$ is independent of $n$.   As $\xi_{\gp, n}^i$ has order $n/(n, i)$, the sum over $i$ with $(n, i)= d < n$ first occurs in the sum for $\E_{n/d}$.
Further,
the inner sum with $i=n$ contributes to the factors of the zeta function of $\E_n$ mod $\gp$ other than $Q(\E_n;\gp,T)$, and the sum in (\ref{ptsonEn}) over $1 \le i < n$ counts the number of $k_\gp$-rational points on the transcendental cycles mod $\gp$, hence it is equal to the coefficient of the second highest order term in $Q(\E_n;\gp,T)$. This shows that, inductively, for $\gp$ not dividing $\ell n$, we have
$$ \Tr \rho_{n,\ell} |_{G_K}(\Frob_\gp) = \Tr \tilde \rho_{n, \ell, t} |_{G_K}(\Frob_\gp) = \sum_{1 \le i < n} ~\sum_{x, y \in k_\gp} \xi_{\gp,n}^i(f_n(x,y)).$$
Moreover, the sum over $1 \le i < n$ with $(i, n)=1$ counts those points on $\E_n$ which are not contained in $\E_d$ with $d$ dividing $n$ properly, so it is the trace of the ``new'' part of $\tilde \rho_{n, \ell, t}$ evaluated at $\Frob_\gp$, i.e.,
\begin{eqnarray}\label{eq:tracerho}
 \Tr \rho_{n,\ell}^{new}|_{G_K}(\Frob_\gp) =  \Tr \tilde \rho_{n, \ell, t}^{new}|_{G_K}(\Frob_\gp) = \sum_{\overset{1 \le i < n} {(n, i)=1}} ~\sum_{x, y \in k_\gp} \xi_{\gp,n}^i(f_n(x,y)).
\end{eqnarray} Both integers are independent of the auxiliary prime $\ell$.

In \cite{DFLST} the authors considered certain families of generalized Legendre curves whose associated Galois representations have a similar decomposition as a sum of new parts. It was shown there that each new part restricted to the Galois group of a suitable cyclotomic field further decomposes into a direct sum of degree-$2$ representations whose trace at the Frobenius elements at unramified places are in fact character sums occurring in counting rational points of the underlying curve over the residue fields. This result is further extended in
\cite{WIN3X} to more general curves. In the theorem below we prove that analogous result holds for restrictions of our 2-dimensional representations $\sigma_{n,\ell,i}$ to finite index subgroups of $G_{\Q(\zeta_n)}$. The argument below follows \S6.3 of \cite{WIN3X}.

\begin{theorem}\label{thm:trace} Fix $n \ge 2$ and $\ell$. Let $K$ be a finite extension of $\Q(\zeta_n)$. Then $\sigma_{n, \ell, i} |_{G_K}$ for $1 \le i \le n-1$ are unramified at each place $\gp$ of $K$ not dividing $\ell n$ and satisfy

$$ \Tr ~\sigma_{n, \ell, i} |_{G_K} (\Frob_\gp)=   \sum_{x, y \in k_\gp} \xi_{\gp,n}^{i}(f_n(x,y)).$$
\end{theorem}

\begin{proof}  First we remark that it suffices to prove the theorem for all pairs $\{n, i\}$ with $(i, n) = 1$. This is because for the case $(i, n) = d >1$, we have $\sigma_{n, \ell, i} \cong \sigma_{n/d, \ell, i/d}$ and $\xi_{\gp, n}^i (f_n(x, y))= \xi_{\gp, n/d}^{i/d}(f_{n/d}(x, y))$ as observed before so that the statement in this case is the same as that with the pair $\{n, i\}$ replaced by $\{n/d, i/d\}$. Therefore we assume $(i, n) = 1$ in the argument below.

Let $\gp$ be a place of $K$ not dividing $n \ell$ with residue field $k_{\gp}$. Then  $\rho_{n, \ell}|_{G_K}$ and hence $\sigma_{n, \ell, i}|_{G_K}$ are unramified at $\gp$ for all $i$. Choose an element $c \in \Z_{K}$ such that $\xi_{\gp, n}(c) =\zeta_n$ is a primitive $n$th root of unity. Then $F = K(\sqrt[n]{c})$ is an abelian extension of $K$ of degree at most $n$ since $K$ contains $\zeta_n$. It follows from (\ref{eq:msymbol}) that the Frobenius element $\Frob_\gp \in \gal(F/K)$, being the inverse of the arithmetic Frobenius $\AF_\fp$ at $\gp$, maps $\sqrt[n]{c}$ to $\xi_{\gp, n}(c)^{-1} \sqrt[n]{c} = \zeta_n^{-1} \sqrt[n]{c}$, hence it has order $n$. Therefore $F$ is a cyclic degree $n$ extension of $K$ and $\AF_\gp$ generates $\gal(F/K)$. The dual of $\gal(F/K)$ is generated by the character $\xi_c$ satisfying $\xi_c( \AF _{\gp})  =\left ( \frac{c}{\fp}\right)_n = \zeta_n$.
\medskip

For each integer $r \ge 0$, consider the twist $\T_{n, c, r}$ of $\E_n$ over $K$ by $c^r$ defined by

$$\T_{n, c, r}~:~ s^n = c^r f_n(x, y).$$
Note that $\T_{n,c,0} = \E_n$.
Denote by $\tilde \rho_{n, \ell, c, r, t}$ the $\ell$-adic $G_{K}$-action on the transcendental lattice of $\T_{n, c, r}$ and by $\tilde \rho_{n, \ell, c, r, t}^{new}$ its new part. The same argument as before shows that
\begin{eqnarray}
 \Tr \tilde \rho_{n, \ell, c, r, t}^{new}(\Frob_\gp) &=& \sum_{1 \le i < n, ~(n, i)=1} ~\sum_{x, y \in k_\gp} \xi_{\gp,n}^{i}(c^r f_n(x,y)) \\
&=&  \sum_{1 \le i < n, ~(n, i)=1} ~ \xi_{\gp,n}^{i}(c^r) \sum_{x, y \in k_\gp} \xi_{\gp,n}^{i}(f_n(x,y)).
\end{eqnarray}
The map $T~: ~(x, y, s) \mapsto (x, y, \sqrt[n]{c}s)$ yields an isomorphism over $F$ from $\T_{n,c,r}$ to $\T_{n,c,r+1}$ for $r \ge 0$.  Therefore $\tilde \rho_{n, \ell, c, r,t}|_{G_F} \cong \rho_{n, \ell} |_{G_F}$ and the same holds for their new part:
\begin{eqnarray}\label{equaloverGF}
 \tilde \rho_{n, \ell, c, r,t}^{new}|_{G_F} \cong \rho_{n, \ell}^{new} |_{G_F}.
\end{eqnarray}On the representation spaces, $T$ induces a map $T^*: \tilde \rho_{n, \ell, c, r+1, t}^{new}\rightarrow \tilde \rho_{n, \ell, c, r, t}^{new}.$
Further, the map $\zeta: (x, y, s) \mapsto (x, y, \zeta_n^{-1} s)$ is an automorphism on $\T_{n, c, r}$ defined over $K \supset \Q(\zeta_n)$, hence it induces an operator $\zeta^*$ of order $n$ on the representation space of $\tilde \rho_{n,\ell,c,r,t}$. This operator decomposes $\tilde \rho_{n, \ell, c, r,t}^{new}$ into a direct sum of $\phi(n)$ subrepresentations $\tau_{n, \ell, c, r, i}$, where $1 \le i \le n-1$ and $(n, i) = 1$, acting on the eigenspace of $\zeta^*$ with eigenvalue $\zeta_n^i$. Since $\zeta$ and $T$ commute over $F$,  $T^*$ yields an isomorphism from the  $\zeta_n^i$-eigenspace of $\zeta^*$ on $\tilde \rho_{n, \ell, c, r+1, t}^{new}$ to that on $\tilde \rho_{n, \ell, c, r, t}^{new}$. Combined with (\ref{equaloverGF}), we obtain
$$\sigma_{n, \ell, i}|_{G_F} \cong \tau_{n, \ell, c, r, i} |_{G_F}$$
for all $1 \le i \le n-1$ coprime to $n$.

On the other hand, for $(x, y, s) \in \T_{n,c,r}( \overline {\Q})$, by \eqref{eq:msymbol} we have
\begin{multline*}\Frob_{\gp,r+1} \circ T(x, y, s) = \Frob_{\gp,r+1}(x, y, \sqrt[n]{c}s)= (\Frob_{\gp,r}(x), \Frob_{\gp,r}(y), \zeta_n^{-1} \sqrt[n]{c}~ \Frob_{\gp,r}(s))\\ =T \circ \zeta \circ \Frob_{\gp,r}(x,y,s).\end{multline*} Here, for the sake of clarity, we put subscripts on $\Frob_{\fp}$ to keep track of the spaces on which it acts. Therefore, on the space of $\tilde \rho_{n,\ell, c, r, t}^{new}$ we have $T^*\circ \Frob_{\fp,r+1}\circ (T^*)^{-1}  = \Frob_{\fp,r} \circ \zeta^*$. Here, by abuse of notation, we use $\Frob_{\fp,r}$  to denote the action on $\tilde \rho_{n,\ell, c, r, t}^{new}$ induced by the geometric  Frobenius at $\gp$.
By construction, $\zeta^*$  acts on the representation spaces of $\tau_{n,\ell, c, r,i}$ and $\tau_{n,\ell, c, r,i} |_{G_F}$ via multiplication by $\zeta_n^i = \xi_{\gp, n}(c)^{i}$; this shows that
$$ \Tr ~\tau_{n, \ell, c, r+1, i}(\Frob_{\gp,r+1}) = \xi_{\gp, n}(c)^{i}~ \Tr ~\tau_{n, \ell, c, r, i}(\Frob_{\gp,r})$$
and recursively this gives

$$ \Tr ~\tau_{n, \ell, c, r, i}(\Frob_{\gp,r}) = \xi_{\gp, n}(c^r)^{i} ~\Tr ~\sigma_{n, \ell, i} |_{G_K}(\Frob_{\gp}).$$
Consequently we obtain, for $r \ge 0$,
\begin{eqnarray}
 \Tr \tilde \rho_{n, \ell, c, r, t}^{new}(\Frob_{\gp,r}) = \sum_{1 \le i < n, ~(n, i)=1} \xi_{\gp,n}(c^r)^{i}~ \Tr ~\sigma_{n, \ell, i} |_{G_K}(\Frob_\gp).
\end{eqnarray}

Compare this with (5.7) for $0 \le r \le \phi(n)-1$. We regard both as a system of $\phi(n)$ linear equations whose coefficient matrix is a nonsingular Vandermonde matrix $\big(\xi_{\gp,n}(c^r)^{ i} \big)_{\overset{0 \le r  \le \phi(n)-1} {1 \le i <n,(n, i)=1}}$
 (because $\xi_{\gp, n}(c) = \zeta_n$ has order $n$). Hence the system has a unique solution, from which it follows that
$$ \Tr ~\sigma_{n, \ell, i}|_{G_K} (\Frob_\gp)=   \sum_{x, y \in k_\gp} \xi_{\gp,n}^{i}(f_n(x,y)),$$
as desired.

\end{proof}

Applying Theorem \ref{thm:trace} to $K = \Q(\zeta_n)$, we get that, for each place $\gp$ of $\Q(\zeta_n)$ not dividing $n \ell$,  $\Tr ~\sigma_{n, \ell, i} (\Frob_\gp)\in \Q(\zeta_n)$ is independent of $\ell$. Further, for such $\gp$  there is a quadratic  extension $K(\gp)$ of $\Q(\zeta_n)$ which has only one place $\wp$ above $\gp$. Applying Theorem \ref{thm:trace} to $K = K(\gp)$, we conclude that $\Tr ~\sigma_{n, \ell, i} |_{G_{K(\gp)}} (\Frob_\wp)\in \Q(\zeta_n)$ is also independent of $\ell$. The same holds for $\det \sigma_{n, \ell, i} (\Frob_\gp) = \frac{1}{2} \big((\Tr ~\sigma_{n, \ell, i} (\Frob_\gp))^2 - \Tr ~\sigma_{n, \ell, i} |_{G_{K(\gp)}} (\Frob_\wp) \big)$.
Combined, this shows that for all finite places $\gp$ of $\Q(\zeta_n)$ not dividing $n$, the characteristic polynomial of $\sigma_{n, \ell, i} (\Frob_\gp)$ is independent of $\ell$ not divisible by $\gp$.
In view of the definition of a compatible system of Galois representations in \S \ref{subsection:CS},  this proves

\begin{cor} \label{cor-compatible}
 For each $1 \le i \le n-1$, the semi-simplifications of $\sigma_{n,\ell,i}$ form a compatible system of $\ell$-adic representations of $G_{\Q(\zeta_n)}$. Further, all $\Tr ~\sigma_{n, \ell, i}$ are $\Q(\zeta_n)$-valued and $\Tr ~\sigma_{n, \ell, i}$ is obtained from $\Tr ~\sigma_{n, \ell, 1}$ with $\zeta_n$ replaced by $\zeta_n^i$.
\end{cor}

Another consequence of Theorem \ref{thm:trace} is the following character sum estimate, resulting from the
fact that the $2$-dimensional representation $\sigma_{n, \ell, i}$ is contained in the second \'etale cohomology of $\E_n$ and the Riemann Hypothesis for the reduction of $\E_n$ at a good place holds.

\begin{cor}\label{cor:5.4} Let $K$ be a finite extension of $\Q(\zeta_n)$. Let $\gp$ be a finite place of $K$ not dividing $n$ whose  residue field $k_\gp$ has cardinality $N\gp$. For $1 \le i \le n-1$ coprime to $n$ we have

$$ |\sum_{x, y \in k_\gp} \xi_{\gp,n}^{i}(f_n(x,y))| \le 2 N\gp.$$

\end{cor}

 At a place $\gp$ of $\Q(\zeta_n)$ not dividing $n$ with residue field $k_\gp$, we can express

\begin{eqnarray*}
\Tr \sigma_{n,\ell, i} (\Frob_\gp) &=& \sum_{x, y \in k_\gp}\xi_{\gp,n}^i(f_n(x,y)) \\
&=& \sum_{x, y \in k_\gp} \xi_{\gp,n}^{-i}(x)\xi_{\gp,n}^i (1-x) \xi_{\gp,n}^{-i}(y) \xi_{\gp,n}^i(1-y) \xi_{\gp,n}^{-i} (1-xy),
\end{eqnarray*} which,  by Corollary 3.14 (i) of Greene \cite{Greene}, is equal to $|k_\gp|^2$ times the hypergeometric function $_3F_2\left({\xi_{\gp,n}^i, ~\xi_{\gp,n}^{-i}, ~\xi_{\gp,n}^{-i} \atop 1, ~1} | 1 \right)$ over the finite field $k_\gp$. Using the identities (4.26) and (4.23) \emph{loc. cit.}, we arrive at the relation
\begin{eqnarray}\label{eq:cxconj}
\sum_{x, y \in k_\gp} \xi_{\gp,n}^i(f_n(x,y)) = \Big(\frac{-1}{\gp}\Big)_n^i \sum_{x, y \in k_\gp} \xi_{\gp,n}^{n-i}(f_n(x,y)),
\end{eqnarray}
for all $1 \le i < n$ and finite places $\gp$ of $\Q(\zeta_n)$ not dividing $n$. By (\ref{eq:msymbol}), the map $\Frob_\gp \mapsto  \Big(\frac{-1}{\gp}\Big)_n$ is a character $\xi_{n, -1}$ of $G_{\Q(\zeta_n)}$.
Combined with Theorem \ref{thm:trace} this gives
\begin{eqnarray}\label{eq:twist}
(\sigma_{n, \ell, i} )^{ss} \cong  (\sigma_{n, \ell, n-i} )^{ss}\otimes \xi_{n,-1}^i.
\end{eqnarray} Here $\sigma^{ss}$ denotes the semi-simplification of the representation $\sigma$.
The kernel of $\xi_{n,-1}$ is $G_{\Q(\zeta_{2n})}$. When $n$ is odd, $\Q(\zeta_{2n}) = \Q(\zeta_{n})$ and hence $\xi_{n,-1}$ is trivial; while for $n$ even, $\Q(\zeta_{2n})$ is a quadratic extension of $\Q(\zeta_n)$ so that $\xi_{n,-1}$ has order $2$.

{Since a semi-simple representation is determined by its trace, we summarize the above discussion below.

\begin{prop}\label{nodd}
For $(i, n)=1$, $\sigma_{n, \ell, i}^{ss} \cong  \sigma_{n, \ell, n-i}^{ss} \otimes \xi_{n,-1}^i$. Consequently $\sigma_{n, \ell, i}^{ss}$ and $\sigma_{n, \ell, n-i}^{ss}$  are equivalent when restricted to $G_{\Q(\zeta_{2n})}$. Further, for $n \ge 3$ odd, we have $ \sigma_{n, \ell, i}^{ss} \cong \sigma_{n, \ell, n-i}^{ss}$.
\end{prop}

\begin{remark}\label{rmk:A} Proposition \ref{nodd} is derived using identities on character sums given in \cite{Greene}. Another way to get the relation between $\sigma_{n,\ell,i}$ and $\sigma_{n,\ell, n-i}$ is to use the symmetry on the modular curve $X_n$ arising from the operator $A=\displaystyle{\left(\begin{matrix}-2 & -5\\1 & 2\end{matrix} \right)} \in \Gamma^0(5)$ which normalizes $\Gamma_n$. By choosing the Hauptmodul $t = E_1/E_2$ with $E_2 = E_1 |_A$ on $X_{\Gamma^1(5)}$, $A$ maps $t$ to $-1/t$ and $t_n$ to $\zeta_{2n}/t_n$. The relation $A \zeta = \zeta ^{-1}A$ on $X_n$ gives rise to the isomorphism $\sigma_{n, \ell, i} \cong \sigma_{n, \ell, n-i}$  for $n$ odd and $\sigma_{n, \ell, i}|_{G_{\Q(\zeta_{2n})}} \cong \sigma_{n, \ell, n-i} |_{G_{\Q(\zeta_{2n})}}$ for $n$ even.
\end{remark}
\subsection{A proof of Theorem \ref{thm:potentialauto}}  First we deal with reducible $\sigma_{n, \ell, i}$.

\begin{lemma}\label{lem-reducible} $\rho_{n,\ell}^{new}$ is automorphic if there is some $1 \le i \le n-1$ coprime to $n$ such that either
\begin{enumerate}

\item $\sigma_{n, \ell, i}$ is reducible, or

\item $\sigma_{n,\ell,i}$ is irreducible and $\sigma_{n, \ell,i} |_{G_{\Q(\zeta_{2n})}}$ is reducible.
\end{enumerate}

\end{lemma}

\begin{proof} By Lemma \ref{lem-structrue}(1), $\sigma_{n,\ell,i}$ is $2$-dimensional and $\rho_{n,\ell}^{new} = \Ind_{G_{\Q(\zeta_n)}}^{G_\Q} \sigma_{n , \ell, i}$.  We divide the proof into two cases according to the assumptions.

(1)  $\sigma_{n, \ell, i}$ is reducible. Then it contains a $1$-dimensional subrepresentation $\chi_1$
 and its semi-simplification decomposes as $\chi_1 \oplus \chi_2$, where $\chi_1$ and $\chi_2$ are geometric characters of $G_{\Q(\zeta_n)}$. Then $\alpha := \Ind_{G_{\Q(\zeta_n)}}^{G_\Q} \chi_1$ is a $\phi(n)$-dimensional subrepresentation of $\rho_{n,\ell}^{new}$. As $\chi_1$ is automorphic and $\Q(\zeta_n)$ is a finite abelian extension of $\Q$, $\alpha$ is automorphic by automorphic induction. The same holds for the quotient $\beta := \rho_{n, \ell}^{new}/\alpha \cong \Ind_{G_{\Q(\zeta_n)}}^{G_\Q} \chi_2$. This proves that $\rho_{n, \ell}^{new}$ is automorphic.

(2) By the assumption on $\sigma_{n,\ell,i}$ and Theorem \ref{thm:Clifford}, $\sigma_{n, \ell,i} |_{G_{\Q(\zeta_{2n})}} = \xi_1 \oplus \xi_2$ for two regular characters $\xi_1$ and $\xi_2$ of $G_{\Q(\zeta_{2n})}$. Note that in this case $\Q(\zeta_{2n})$ is a quadratic extension of $\Q(\zeta_n)$ and $n$ is even. Let $g$ be an element in $G_{\Q(\zeta_n)} \smallsetminus G_{\Q(\zeta_{2n})}$. Then $\xi_1^g = \xi_1$ or $\xi_2$.

	Case (2.1) $\xi_1^g = \xi_2$. Then $\sigma_{n,\ell,i} = \Ind_{G_{\Q(\zeta_{2n})}}^{G_{\Q(\zeta_{n})}}\xi_1$ so that $\rho_{n,\ell}^{new} = \Ind_{G_{\Q(\zeta_{2n})}}^{G_\Q} \xi_1$ is automorphic by automorphic induction.

Case (2.2) $\xi_1^g = \xi_1$.  Then $\xi_2^g = \xi_2$ so that both $\xi_1$ and $\xi_2$ extend to characters of $G_{\Q(\zeta_{n})}$. This contradicts the irreducibility of $\sigma_{n,\ell,i}$.
\end{proof}

In view of the Lemma above, we shall assume that all $\sigma_{n, \ell, i} |_{G_{\Q(\zeta_{2n})}}$ with $(i, n)=1$ are absolutely irreducible for the rest of the proof.

When $n=2$, $\rho_{2, \ell}$ is $2$-dimensional and odd, hence is automorphic. Moreover, $\rho_{2,\ell} = \sigma_{2,\ell,1} \cong \xi_{2,-1} \otimes \rho_{2,\ell}$ by Proposition \ref{nodd}. This implies that $\rho_{2,\ell}$ has CM by $\Q(\sqrt{-1})$, as observed in \cite{Beukers-Stienstra}.

Now we prove the theorem for $n \ge 3$, in which case $\phi(n)$ is even. With $n$ fixed, denote by $F_n: = \Q(\zeta_n)^+$ the totally real subfield of the cyclotomic field $\Q(\zeta_n)$. The restriction of the complex conjugation $c$ to $\Q(\zeta_n)$ generates $\gal(\Q(\zeta_n)/\Q(\zeta_n)^+)$.  It follows from Lemma \ref{lem-structrue} and Proposition \ref{nodd} that
$$\tau_{n,\ell,i} :=  \Ind_{G_{\Q(\zeta_n)}}^{G_{F_n}} \sigma_{n , \ell , i} = \Ind_{G_{\Q(\zeta_n)}}^{G_{F_n}} \sigma_{n , \ell , n-i} $$ is potentially 2-isotypic over $\Q(\zeta_{2n})$ and
$$\displaystyle \rho_{n, \ell}^{new} \cong \text{Ind}_{G_{F_n}}^{G_\Q} \tau_{n,\ell, i}.    $$

To prove the potential automorphy of $\rho_{n,\ell}^{new}$ for $n \ge 3$, our strategy is to use Proposition \ref{prop-save}. For this purpose, we need to show, for each $1 \le i \le (n-1)/2, (i, n) = 1$, the existence of $2$-dimensional $\ell$-adic representations $\eta_{n , \ell , i}$ and $\eta _{n , \ell , n-i }$ of $G_{F_n}$ so that their restrictions to $G_{\Q(\zeta_n)} $ are finitely  projectively equivalent to $ \sigma_{n , \ell , i}$ and $\sigma_{n , \ell , n-i}$, respectively.
To do this, we divide the argument into two cases, according as $\tau_{n , \ell, i}$ reducible or irreducible. Observe from Lemma \ref{lem-structrue}(1)  that,  for any $g \in G_\Q$, $\tau_{n, \ell, i }^g = \tau_{n , \ell , \e(g)i}$. So the $\tau_{n , \ell , i}$'s will be simultaneously reducible or irreducible.

(i) $\tau_{n , \ell , i}$ is reducible. This includes all odd $n \ge 3$ because in this case $\sigma_{n,\ell,i} \cong \sigma_{n,\ell,i}^c = \sigma _{n , \ell , n-i}$ by Proposition \ref{nodd}. Since $\sigma_{n, \ell , i}$ and $\sigma _{n , \ell , n-i}$ are assumed to be  irreducible, this forces the semi-simplificiation of $\tau_{n , \ell, i}$ to be $\eta_{n ,\ell,  i}\oplus \eta_{n , \ell, n-i}$ so that $\eta_{n , \ell , i}|_{G_{\Q(\zeta_n)}} \simeq \sigma_{n , \ell, i}$  and $\eta_{n,\ell, n-i} = \eta_{n,\ell,i} \otimes \chi$ with $\chi$ the quadratic character associated to $\Q(\zeta_n)/F_n$. (In fact  $\tau_{n,\ell,i} = \Ind_{G_{\Q(\zeta_n)}}^{G_{F_n}} \sigma_{n,\ell, i} = \eta_{n,\ell,i} \oplus \eta_{n,\ell,i} \otimes \chi$ in this case.)
Now we can apply Corollary \ref{co-save} to conclude that $\rho_{n , \ell}^{new}$ is potentially automorphic, and in fact automorphic when $F_n = \Q$.

(ii) $\tau_{n , \ell , i}$ is irreducible. Then $n \ge 4$ is even. By Proposition \ref{nodd}, each $\tau_{n,\ell,i}$ is 2-isotypic over $\Q(\zeta_{2n})$ since $\sigma_{n,\ell,i}|_{G_{\Q(\zeta_{2n})}}$ is assumed to be irreducible.  According to Theorem \ref{thm:Clifford}, there exists a $2$-dimensional $\ell$-adic representation  $\eta_{n , \ell , i}$ of $G_{F_n}$ so that $\eta_{n , \ell , i}|_{G_{\Q(\zeta_{2n} )}}$ is finitely projectively  equivalent to $\sigma_{n , \ell , i}|_{G_{\Q(\zeta_{2n} )}}$. We also have $\eta_{n ,\ell , i}|_{G_{\Q(\zeta_{2n})}}\simeq \eta^c_{n ,\ell , i}|_{G_{\Q(\zeta_{2n})}}$ finitely projectively equivalent to $\sigma ^c_{n, \ell , i}|_{G_{\Q(\zeta_{2n} )}}= \sigma_{n ,\ell , n-i}|_{G_{\Q(\zeta_{2n} )}}$. So we may choose $\eta_{n, \ell, n-i} = \eta_{n, \ell, i}$. Thus the requirement of Corollary \ref{co-save} is satisfied and thus  $\rho_{n ,\ell}^{new}$ is potentially automorphic.

When $n=3, 4, 6$, the field $F_n=\Q$, $\rho_{n,\ell}^{new}$ is $4$-dimensional and it is $2$-isotypic over $\Q(\zeta_{2n})$. We can also conclude the automorphy of $\rho_{n,\ell}^{new}$ from Theorem \ref{thm:main2}. For $n=3$, the argument in (i) above shows that $\rho_{3,\ell}^{new}$ is the sum of two degree-$2$ automorphic representations which differ by the quadratic twist $\chi_{-3}$ attached to $\Q(\sqrt{-3})/\Q$, as shown in \cite{LLY}. When $n=4$ and $6$, Clifford theory (Theorem \ref{thm:Clifford}) implies that $\rho_{n,\ell}^{new}$ corresponds to an automorphic representation of $\GL_2 \times \GL_2$ over $\Q$, as explained in \cite{ALLL}. When $n=5$, as remarked before, the group $\Gamma_5$ is isomorphic to the congruence subgroup $\Gamma(5)$, hence $\rho_{5, \ell}$ is automorphic.

To complete the proof of Theorem \ref{thm:potentialauto}, it remains to show

\begin{prop}\label{prop-reducible}

If $\sigma_{n, \ell, i}$ is potentially reducible for some $i$ coprime to $n$, then $\rho_{n,\ell}^{new}$ is automorphic.

\end{prop}

\begin{proof}  In view of Lemma \ref{lem-reducible}, we may assume that $\sigma_{n , \ell , i} |_{G_{\Q(\zeta_{2n})}}$ is irreducible and it is potentially reducible. From the discussion of cases (i) and (ii) above,    there exists a representation $\eta_{n, \ell , i}$  of $G_{F_n}$ so that $\eta_{n , \ell , i}|_{G_{\Q(\zeta_{2n})}}$ is finitely projectively equivalent to $\sigma_{n , \ell , i}|_{G_{\Q(\zeta_{2n})}} . $ So $\eta_{n , \ell , i}$ is potentially reducible and hence is  automorphic  by Proposition \ref{prop-auto}. Hence so is $\eta_{n , \ell , i }|_{G_{\Q(\zeta_{2n})}}$, which is $\sigma_{n , \ell , i } |_{G_{\Q(\zeta_{2n})}} $ twisted by a finite character of $G_{\Q(\zeta_{2n})}$. Therefore $\sigma _{n , \ell, i }$ is automorphic, and so is its induction to $G_\Q$. This proves that $\rho_{n, \ell}^{new}$ is automorphic.
\end{proof}

\begin{remark}\label{rm-etachoice} When $\sigma_{n,\ell,i} |_{G_{\Q(\zeta_{2n})}}$ are assumed to be irreducible, we have seen from the above discussion that for each $\ell$ there exists a representation $\eta_{n , \ell , i}$ of
$G_{F_n}$ so that $ \eta_{n , \ell , i}|_{G_{\Q(\zeta_{n})}}$ is finitely projectively equivalent to $\sigma_{n , \ell , i}$.  We claim that $\eta_{n , \ell , i}$ can be chosen to be part of a compatible system. More precisely, there exists a compatible system $\{ E, S , \{Q_\gp (X)\}, \{\tilde \eta_{n ,\lambda, i}\}\}$ of $\ell$-adic Galois representation of $G_{F_n}$ (see \S \ref{subsection:CS}) so that for each $\ell$ there exists a prime $\lambda$  of $E$ over $\ell$ satisfying that $\tilde \eta_{n , \lambda, i} |_{G_{\Q(\zeta_n)}}$ is finitely projectively equivalent to $\sigma_{n , \ell, i}$. In other words, $\eta_{n , \ell, i}$ can be chosen to be some $\tilde \eta_{n , \lambda, i}$. Indeed,  we start with the $\eta_{n, \ell , i}$ for each $\ell$ from the proof of Theorem \ref{thm:potentialauto} without knowing whether they are part of a compatible system. The argument of the proof (in particular, when we use Proposition \ref{prop-save}) implies the existence of at least one large prime $\ell'$ so that $\eta_{n , \ell', i}$ is potentially automorphic. By  Theorem 5.5.1 in \cite{BGGT}, there exists a compatible system $\{ E, S , \{Q_\gp (X)\}, \{\tilde \eta_{n ,\lambda, i}\}\}$ of $\ell$-adic Galois representations of $G_{F_n}$ so that for some  prime $\lambda'$ of $E$ over $\ell'$ we have $\tilde \eta_{n , \lambda', i} = \eta_{n , \ell' , i}$. To see that the system $\{\tilde \eta_{n , \lambda, i} \}$ has the desired property, it suffices to check that, for each prime $\lambda$ of $E$ over $\ell$,  $\tilde \eta_{n , \lambda, i} |_{G_{\Q(\zeta_n)}}$ is finitely projectively equivalent to $\sigma_{n , \ell, i}$. Let $\chi_{n , \ell', i }$ be the character with finite image so that $\eta_{n , \ell', i} |_{G_{\Q(\zeta_n)}} = \sigma _{n , \ell', i}\otimes \chi _{n , \ell', i}$. Since the image of $\chi_{n , \ell' , i}$ in $\overline {\Q}_{\ell'}^\times$ is finite, it isomorphically imbeds in $\overline {\Q}^\times$ so that it forms a compatible system of $G_{\Q(\zeta_n)}$. Denote by $\chi_{n, \ell,i}$ the same character with image viewed in $\overline {\Q}_\ell^\times$. We claim that, for all $\ell$, $\tilde \eta_{n , \lambda, i} |_{G_{\Q(\zeta_n)}} = \sigma _{n , \ell, i} \otimes \chi _{n , \ell, i}$ for each $\lambda$ above $\ell$.  As $\sigma_{n , \ell, i}$ is assumed to be irreducible, it suffices to check the trace of $\Frob_{\gp}$ on both sides for all primes $\gp$ of $\Q(\zeta_n)$ not dividing $n \ell$, for then the semi-simplification of $\tilde \eta_{n , \lambda , i}|_{G_{\Q(\zeta_n)}}$  and hence $\tilde \eta_{n , \lambda, i}|_{G_{\Q(\zeta_n)}}$ will be irreducible and equal to $\sigma _{n , \ell, i} \otimes \chi _{n , \ell, i}$. The claim then follows from the compatibility of three systems: $\{\tilde \eta_{n , \lambda, i}\}$, $\{\chi_{n , \ell, i}\}$ and $\{\sigma_{n , \ell, i}\}$ (by
Corollary \ref{cor-compatible}), and the fact that the equality holds for $\lambda'$ and $\ell'$.
\end{remark}

\subsection{$\tau_{n,\ell,i}$ admits QM}\label{ss:5.7} The concept of quaternion multiplication over a field was introduced in \cite[Definition 3.1.1]{ALLL} for $4$-dimensional representations of $G_\Q$. We extend it to representations of finite index subgroups of $G_\Q$.

\begin{defn}\label{def:QM} Let $F$ be a number field.
A finite-dimensional representation $\rho$ of $G_F$ is said to admit \emph{quaternion multiplication (QM)} if there are two linear operators $J_+$ and $J_-$ acting on the representation space of $\rho$ such that
\begin{itemize}

\item[(1)] $J_\pm^2 = -$ id and $J_+ J_- = - J_- J_+$ (so that $J_\pm$ generate the quaternion group $Q_8$);

\item[(2)] There exist two multiplicative characters $\chi_\pm$ of $ G_F$ of order $\le 2$ such that  for any $g \in G_F$, $$\rho(g)\circ J_\pm =\chi_{\pm}(g) J_\pm\circ \rho(g).$$
\end{itemize}

We say that $\rho$ admits \emph{QM over $L$} if $\chi_\pm$ are trivial on $G_L$.

\end{defn}

Fix a choice of $1 \le i \le n-1$ coprime to $n$.
We claim that the representation $\tau_{n,\ell,i} = \Ind_{G_{\Q(\zeta_n)}}^{G_{F_n}} \sigma_{n , \ell , i}$ of $G_{F_n}$ studied in the previous subsection admits QM over $\Q(\zeta_{2n})$, generalizing the known results for $n=3,4,6$ discussed in \cite{ALLL}. To see this, recall the symmetry $A=\displaystyle{\left(\begin{matrix}-2 & -5\\1 & 2\end{matrix} \right)} \in \Gamma^0(5)$ mapping $t_n$ to $\zeta_{2n}/t_n$ mentioned in Remark \ref{rmk:A}. It induces a map $A$ on the model (\ref{eq:En2}) of $\E_n$
sending $(x, y, s)$ to $(1-x, 1/(1-xy), \zeta_{2n}(1-x)(1-y)x^2y/(s(1-xy)))$. The relation $ \zeta A \zeta = A$ on $\E_n$ gives rise to the relation
$$ \zeta^* A^* \zeta^* = A^*$$
satisfied by the operators $\zeta^*$ and $A^*$ acting on the space of $\rho_{n, \ell}^{new}$. The operator $\zeta^*$ is defined over $\Q(\zeta_n)$ while  $A^*$ is defined over $\Q(\zeta_{2n})$. When $n$ is odd, $\Q(\zeta_{2n}) = \Q(\zeta_n)$ is quadratic over the totally real subfield $F_n$ of $\Q(\zeta_n)$. When $n$ is even, $\Q(\zeta_{2n})$ is a biquadratic extension of  $F_n = \Q(\cos \frac{2 \pi}{n})$ with three quadratic intermediate fields: $F_{2n} = F_n(\cos \frac{2 \pi}{2n})$, $\Q(\zeta_n) = F_n(\zeta_4 \sin \frac{2\pi}{n})$, and  $F_n(\zeta_4 \sin \frac{2\pi}{2n})$. As finite abelian extensions of $\Q$, these three fields are characterized by the primes splitting completely in them, which are respectively $p \equiv \pm1 \mod 2n$, $p \equiv 1 \mod n$, and $p \equiv 1, n-1 \mod 2n$. The primes $p$ splitting completely in $F_n$ are $\equiv \pm 1 \mod n$.

\begin{lemma}\label{F_v} Let $v$ be a prime of $F_n$ dividing an odd prime $p \equiv \pm1 \mod n$ so that $Nv = p$.
Then on $\E_n$ we have $\Frob_v A = \zeta^{(1-p)/2} A \Frob_v$ and $\Frob_v \zeta = \zeta^p \Frob_v$. The actions of $\zeta^*$ and $A^*$ on $\rho_{n,\ell}^{new}$ satisfy
$$ A^* \Frob_v = \Frob_v A^* (\zeta^*)^{(1-p)/2} \quad {\rm and} \quad \zeta^* \Frob_v = \Frob_v (\zeta^*)^p.$$
\noindent Here we retain the same notation for the Frobenius action on $\rho_{n,\ell}^{new}$.
\end{lemma}

\begin{proof} We prove the identities on $\E_n$ by checking the actions of the maps on a point $(x, y, s) \in \E_n(\bar \Q)$. The induced actions on $\rho_{n,\ell}^{new}$ satisfy the relations on $\E_n$ with reversed order because the operators act on a cohomology space. The first identity follows from
\begin{eqnarray*}
\Frob_v A(x, y, s)&=& \Frob_v(1-x, 1/(1-xy), \zeta_{2n}(1-x)(1-y)x^2y/(s(1-xy))) \\
&=& (1-\Frob_v(x), \frac{1}{1 - \Frob_v(xy)}, \zeta_{2n}^p\Frob_v(\frac{(1-x)(1-y)x^2y}{s(1-xy)}))
\end{eqnarray*}

and
\begin{eqnarray*}
\zeta^{(1-p)/2} A \Frob_v(x, y, s) &=& \zeta^{(1-p)/2} A (\Frob_v(x), \Frob_v(y), \Frob_v(s)) \\
&=& \zeta^{(1-p)/2}(1-\Frob_v(x), \frac{1}{1-\Frob_v(xy)},\zeta_{2n}\Frob_v(\frac{(1-x)(1-y)x^2y}{s(1-xy)})) \\
&=& (1-\Frob_v(x), 1/(1-\Frob_v(xy)), \zeta_{2n}^{p-1}\zeta_{2n}\Frob_v(\frac{(1-x)(1-y)x^2y}{s(1-xy)}));
\end{eqnarray*}
while the second identity results from $\zeta(x,y,s) = (x, y, \zeta_n^{-1}s)$ noted before:
$$\Frob_v \zeta(x, y, s) = \Frob_v(x, y, \zeta_n^{-1}s) = (\Frob_v(x), \Frob_v(y), \zeta_n^{-p} \Frob_v(s)) = \zeta^p \Frob_v(x, y, s).$$
\end{proof}

To show that $\tau_{n,\ell,i}$ admits QM, consider the operators

$$B_+:=(1+(\zeta^*)^{-1})A^* \quad {\rm and} \quad
B_-:=(1-(\zeta^*)^{-1})A^*$$

\noindent on the space of $\rho_{n,\ell}^{new}$. They leave invariant $\tau_{n,\ell,i} = \sigma_{n, \ell, i} \oplus \sigma_{n, \ell, n-i}$ since each summand is invariant under $\zeta^*$ and the two summands are swapped by $A^*$. It is straightforward to check, using the relations $\zeta^* A^*\zeta^* = A^*$ and $(A^*)^2 = -I$, that

$$ B_+^2 = -(2 + \zeta^* + (\zeta^*)^{-1}), \quad B_-^2 = -(2 - \zeta^* - (\zeta^*)^{-1}), \quad {\rm and} \quad B_+ B_- = -B_- B_+.$$

\noindent Consequently on $\tau_{n,\ell,i}$, we have $B_+^2 = -(2 + 2 \cos \frac{2i\pi}{n}) = -4 \cos^2 \frac{2 i \pi}{2n}$ and $B_-^2 = -(2 - 2 \cos \frac{2i \pi}{n}) = -4 \sin^2 \frac{2i\pi}{2n}$. Further, $$B:=B_+B_- = \zeta^*- (\zeta^*)^{-1}$$ on $\tau_{n,\ell,i}$ satisfies $B^2 = - 4 \sin^2 \frac{2i\pi}{n}$.

Next we determine the commuting relations between these operators and $G_{F_n}$.
As $A^*$ and $\zeta^*$ commute with $G_{\Q(\zeta_{2n})}$, so do  $B_\pm$ and $B$.

\begin{prop}\label{Bpm} (I) When $n$ is even, $\Q(\zeta_{2n})$ is a biquadratic extension of $F_n$. On $\rho_{n,\ell}^{new}$ we have

(1) $B_+$ commutes with $G_{F_{2n}}$, and $B_+ \rho_{n,\ell}^{new}(g) = - \rho_{n,\ell}^{new}(g) B_+$ for $g \in G_{F_n} \smallsetminus G_{F_{2n}}$;

(2) $B_-$  commutes with $G_{F_n(\zeta_4 \sin \frac{2\pi}{2n})}$, and $B_- \rho_{n,\ell}^{new}(g) = - \rho_{n,\ell}^{new}(g) B_-$ for $g \in G_{F_n} \smallsetminus G_{F_n(\zeta_4 \sin \frac{2\pi}{2n})}$;

(3) $B$ commutes with $G_{\Q(\zeta_n)}$, and $B \rho_{n,\ell}^{new}(g) = - \rho_{n,\ell}^{new}(g) B$ for $g \in G_{F_n} \smallsetminus G_{\Q(\zeta_n)}$.

(II) When $n$ is odd, $\Q(\zeta_{2n}) = \Q(\zeta_n)$ is a quadratic extension of $F_n$. On $\rho_{n,\ell}^{new}$ we have $B_-$ and $B$ commute with $G_{\Q(\zeta_n)}$ and anti-commute with elements in $G_{F_n} \smallsetminus G_{\Q(\zeta_n)}$.

Therefore $\tau_{n,\ell,i}$ admits QM over $\Q(\zeta_{2n})$ with operators $J_+:= \frac {1}{2 \sin \frac{2i\pi}{2n}} B_-$ and $J_-:= \frac{1}{2 \sin \frac{2i\pi}{n}} B$, and $\chi_{\pm}$ being quadratic characters of $G_{F_n}$ with kernels $G_{F_n(\zeta_4 \sin \frac{2\pi}{2n})}$ and $G_{\Q(\zeta_n)}$, respectively.
\end{prop}

\begin{proof} (I) For $n$ even, observe that $\Gal(\Q(\zeta_{2n})/F_n) \cong G_{F_n}/G_{\Q(\zeta_{2n})}$ is a Klein four group consisting of $\Frob_v$ with places $v$ of $F_n$ above any prime $p \equiv 1, -1, n+1, n-1\mod 2n$, respectively. This is because such primes $p$ split completely in  $\Q(\zeta_{2n})$, $F_{2n}$, $\Q(\zeta_n)$, and $F_n(\zeta_4 \sin \frac{2\pi}{2n})$, respectively, and these are the fixed fields of the corresponding $\Frob_v$ in $\Q(\zeta_{2n})$.

Now let $v$ be a prime of $F_n$ dividing an odd prime $p \equiv \pm1 \mod n$, we examine the commuting relation between $\Frob_v$ and $B_\pm$ and $B$. Using Lemma \ref{F_v}, we get
\begin{eqnarray*}
B_{\pm} \Frob_v &=& (1\pm (\zeta^{-1})^*)A^* \Frob_v =  (1\pm (\zeta^{-1})^*) \Frob_v A^* (\zeta^*)^{(1-p)/2}\\
&=& \Frob_v (1 \pm (\zeta^*)^{-p}) A^* (\zeta^*)^{(1-p)/2} = \Frob_v (1 \pm (\zeta^*)^{-p}) (\zeta^*)^{(p-1)/2} A^*.
\end{eqnarray*}

For $p \equiv 1 \mod 2n$, the place $v$ of $F_n$ splits completely in $\Q(\zeta_{2n})$. As observed above,  $B_\pm$ and $B$ commute with such $Frob_v$.
For $p \equiv -1 \mod 2n$, we have
$$ B_+ \Frob_v = \Frob_v (1 + \zeta^*)(\zeta^*)^{-1} A^* = \Frob_v B_+$$
and for $p \equiv n \pm 1 \mod 2n$ it is straightforward to verify $B_+ \Frob_v = - \Frob_v B_+$. This proves (1).

 For (2) we find for $v$ above $p \equiv n-1 \mod 2n$,
$$ B_- \Frob_v = \Frob_v (1 - (\zeta^*))(\zeta^*)^{n/2-1} A^* = - \Frob_v ((\zeta^*)^{-1}-1) = \Frob_v B_-$$
and one checks that for $v$ above $p \equiv -1$ and $n+1 \mod 2n$ we have $B_- \Frob_v = - \Frob_v B_-$.
This proves (2).

For (3) we note that
$$B \Frob_v = (\zeta^* - (\zeta^*)^{-1})\Frob_v = \Frob_v ((\zeta^*)^p - (\zeta^*)^{-p}).$$
Thus for $p \equiv 1 \mod n$ this gives $B \Frob_v = \Frob_v B$  and for $p \equiv -1 \mod n$ this gives $B\Frob_v = - \Frob_v B$, as asserted in (3).

(II) When $n$ is odd, $\Q(\zeta_{2n}) = \Q(\zeta_n)$ so that $\zeta^*$ and $A^*$ commute with $G_{\Q(\zeta_n)}$, and so do $B_\pm$ and $B$. Since $F_{2n} = F_n$ and $F_n(\zeta_4 \sin \frac{2\pi}{2n}) = \Q(\zeta_n)$ in this case,  the above computation shows that $B_-$ and $B$ commute with $G_{\Q(\zeta_n)}$ and anti-commute with elements in $G_{F_n}$ but outside $G_{\Q(\zeta_n)}$, while $B_+$ commutes with $G_{F_n}$.  This completes the proof of the proposition.

\end{proof}

Now we discuss the situation where the automorphy of $\rho_{n,\ell}^{new} = \Ind_{G_{\Q(\zeta_n)}}^{G_\Q} \sigma_{n,\ell,i}$ for $(i, n)=1$ is not yet known. Then $\sigma_{n,\ell,i}$ is strongly irreducible by Proposition \ref{prop-reducible}. Since $\tau_{n,\ell,i} = \Ind_{G_{\Q(\zeta_n)}}^{G_{F_n}} \sigma_{n,\ell,i}$ admits QM over $\Q(\zeta_{2n})$ (Proposition \ref{Bpm}), following the same argument as the proof of Theorem 3.2.1 in \cite{ALLL}, one obtains a finite character $\xi_n$ of $G_{\Q(\zeta_n)}$ such that $\sigma_{n,\ell,i} \otimes \xi_n$ extends to a degree-$2$ representation $\eta_{n,\ell,i}$ of $G_{F_n}$ and $\tau_{n,\ell,i} = \eta_{n,\ell,i} \otimes \Ind_{G_{\Q(\zeta_n)}}^{G_{F_n}} \xi_n^{-1}$. To handle the case that $\tau_{n,\ell,i}$ is irreducible, Case (ii) of the proof of  Theorem \ref{thm:potentialauto} used
Clifford theory  (Theorem \ref{thm:Clifford}) to conclude $\tau_{n,\ell,i} = \eta_{n,\ell,i} \otimes \gamma_{n,\ell,i}$ for a degree-$2$ representation $\eta_{n,\ell,i}$ of $G_{F_n}$ whose restriction to $G_{\Q(\zeta_{2n})}$ differs from $\sigma_{n,\ell,i} |_{G_{\Q(\zeta_{2n})}}$ by a finite character, and a degree-$2$ representation $\gamma_{n,\ell,i}$ of $G_{F_n}$ with finite image. Here using the QM structure, we gain the information that $\gamma_{n,\ell,i}$ can be chosen to be the representation induced from a finite character of $G_{\Q(\zeta_n)}$, and thus is automorphic. When $\tau_{n,\ell,i}$ is reducible, Case (i) of the proof of  Theorem \ref{thm:potentialauto} shows that we can also write $\tau_{n,\ell,i} = \eta_{n,\ell,i} \otimes \gamma_{n,\ell,i}$ with $\gamma_{n,\ell,i}$ being the sum  of two finite characters of $G_{F_n}$, in other words, $\Ind_{G_{\Q(\zeta_n)}}^{G_{F_n}} \xi_n^{-1}$ decomposes into the sum of two finite characters, which is also automorphic. The same argument as in Remark \ref{rm-etachoice} shows that we may assume $\eta_{n,\ell,i}$ to be part of a compatible system.

We summarize the conclusion of this section in the following remark.

\begin{remark}
For $n \ge 3$, $(i,n)=1$ and a prime $\ell$, the representation $\tau_{n,\ell,i}$ admits QM over $\Q(\zeta_{2n})$. By Theorem \ref{thm-try} the degree-$2$ representation $\eta_{n,\ell,i}$ of $G_{F_n}$ occurring in $\tau_{n,\ell,i}$ as above  is totally odd, with Hodge-Tate weights $0$ and $-2$, and potentially automorphic. Moreover, it is  automorphic if $F_n = \Q$ or it is potentially reducible. Further $\eta_{n,\ell,i}$ can be chosen to be part of a compatible system and $\tau_{n,\ell,i}$ is automorphic if and only if $\eta_{n,\ell,i}$ is. Therefore the automorphy of $\eta_{n,\ell,i}$ for all $(n,i) = 1$ will imply the automorphy of $\rho_{n,\ell}^{new}$ by automorphic induction.
\end{remark}

\subsection{Remarks on other families of Scholl representations attached to noncongruence subgroups} The noncongruence groups $\Gamma_n$ for $n \ne 5$ by construction are finite index normal subgroups of $\Gamma^1(5)$. The group $\G^1(5)$ is one of the  6 isomorphism classes of torsion-free index-12 subgroups in $PSL_2(\Z)$. In \cite{L5}, the authors constructed similar noncongruence  subgroups from other torsion-free index-12 subgroups, such as $\G_1(6)$.  Following Beauville, the authors in \cite{L5} use the equation $(xy+yx+zx)(x+y+z)= txyz/9$ for the universal family of elliptic curves with an order $6$ torsion point (see Table 6 of \cite{L5}) where $t =\frac{\eta(6z)^4\eta(z)^8}{\eta(3z)^8\eta(2z)^4}$ is a Hauptmodul for  $\Gamma_1(6)$. This plays the same role as \eqref{eq:5.4} in the $\G^1(5)$ case. When one considers  cyclic cover  of the modular curve of $\G_1(6)$ by replacing $t$ by $t_n =\sqrt[n]{t}$ and denotes by $\tilde \G_n$ the corresponding index-$n$ subgroup of $\G_1(6)$, then $S_3(\tilde \G_n)$ is of dimension $n-1$ and
{the modular curve $X_{\tilde \Gamma_n}$ and the universal elliptic curve over it} admit the  automorphism $\zeta: t_n\mapsto \zeta_n^{-1} t_n$. Because of $\zeta$ the Scholl representations associated with $S_3(\tilde \Gamma_n)$, when restricted to $G_{\Q(\zeta_n)}$,  also decompose into a sum of degree-$2$ factors $\sigma_{n,\ell,i}$ for $1\le i<n$. Using the same argument, one has $\sigma_{n,\ell,i}^c\cong\sigma_{n,\ell,n-i}$ and they are swapped by the  $W_3$ operator of $\G_1(6)$ which plays the role of $A$ for $\G^1(5)$.  On the universal elliptic curve,  $W_3$ gives rise to an isogeny map, see \S5.1 of \cite{L5}.  So, one can draw the same (potential) automorphy conclusion for the corresponding Scholl representations. It is worth pointing out that all cusps of $\G_1(6)$ are defined over $\Q$, and  any cusp can be sent to $\infty$ via one of the Atkin-Lenher involutions (see Table 12 of \cite{L5} for the corresponding linear transformation on $t$). Consequently one can construct other infinite families of cyclic subgroups of $\G_1(6)$ ramified only at two cusps which give rise to Galois representations of $G_\Q$ for which similar (potential) automorphy  conclusions can be drawn.

\renewcommand{\baselinestretch}{1.1}


\begin{thebibliography}{}

\bibitem[AC]{AC89} J.~Arthur and L.~Clozel, \emph{Simple algebras, base change, and the advanced theory of the trace formula}.
Annals of Math. Studies, no.120, Princeton University Press, 1989.

\bibitem[AL$^3$]{ALLL} A.O.L.~Atkin, W.-C. W.~Li, T.~Liu, and L.~Long, \emph{Galois representations with quaternion multiplications associated to noncongruence modular forms}. Trans. Amer. Math. Soc. 365 (2013), no. 12, 6217--6242.

\bibitem[ALL]{ALL}A.O.L.~Atkin, W.-C. W.~Li,  and L.~Long,  \emph{On Atkin and Swinnerton-Dyer congruence relations}. II. Math. Ann. 340 (2008), no. 2, 335--358.

%\bibitem[BGHT]{BGHT} T.~Barnet-Lamb, D.~Geraghty, M.~Harris, and R.~Taylor,  \emph{A family of Calabi-Yau varieties and potential automorphy II}.   Publ. Res. Inst. Math. Sci. 47 (2011), no. 1, 29--98.

\bibitem[BGGT]{BGGT} T.~Barnet-Lamb, T.~Gee.   D.~Geraghty,  and R.~Taylor,  \emph{Potential automorphy and change of weight}.  Annals of Math  179 (2014), 501--609.
\bibitem[Bou]{Trace}  N.~Bourbaki, \emph{\'El\'ements de Math\'ematique. Alg\'ebre.} Chapitre 8. Modules et anneaux semi-simples.  Second revised edition of the 1958 edition, Springer, Berlin, 2012.

\bibitem[CG]{Cal11}F.~Calegari and T.~Gee, \emph{Irreducibility of automorphic Galois representations of GL(n), $n$ at most 5}. Annales de l'Institut Fourier, Vol. 63 no. 5 (2013), p. 1881--1912.

\bibitem[Cli]{Clifford37}
A.~H. Clifford, \emph{Representations induced in an invariant subgroup}.  Ann.
  of Math. (2) \textbf{38} (1937), no.~3, 533--550.


\bibitem[DFLST]{DFLST}A.~Deines, J.G.~Fuselier, L.~Long,  H.~Swisher, and F.-T.~Tu, \emph{Generalized Legendre curves and Quaternionic Multiplication}. Journal of Number Theory,  161 (2016), 175--203.







  \bibitem[Del]{Deligne}P.~Deligne,
\emph{Travaux de Shimura}.  S\'eminaire Bourbaki, 23\`eme ann\'e (1970/71), Exp. No. 389, pp. 123--165. Lecture Notes in Math.,  (1971), Vol. 244, Springer, Berlin.

\bibitem[DFG]{DFG} F.~Diamond, M.~Flach, and L.~Guo, \emph{The Tamagawa number conjecture of adjoint motives of modular forms.} Ann. Sci. Ecole Norm. Sup. (4) vol. 37 (2004), no. 5, 663-727.


\bibitem[ES]{Elkies-Schutt}N. D.~Elkies and M.~Sch\"utt, \emph{Modular forms and K3 surfaces}. Adv. Math. 240 (2013), 106--131.

\bibitem[FHLRV]{L5}L.~Fang, J.W.~Hoffman, B.~Linowitz, A.~Rupinski and  H.~Verrill,  \emph{Modular forms on noncongruence subgroups and {A}tkin-{S}winnerton-{D}yer relations.} Experiment. Math. 19 (2010), no. 1, 1--27.

\bibitem[Far]{Fargues} L.~Fargues, \emph{Motives and automorphic forms: the (potentially) abelian case}. Note, avaliable at {\tt{https://webusers.imj-prg.fr/~laurent.fargues/Motifs\_abeliens.pdf}}

\bibitem[FO]{FO} J. M.~Fontaine and Y.~Ouyang, \emph{Theory of $p$-adic Galois Representations}. Book,  avaliable at {\tt{http://www.math.u-psud.fr/$\sim$fontaine/galoisrep.pdf}}

\bibitem[FLRST]{WIN3X}J.G.~Fuselier, L.~Long, R.~Ramakrishna, H.~Swisher, and F.-T.~Tu, \emph{Hypergeometric functions over finite fields.} arxiv:1510.02575, (2016).

%\bibitem[Gee09]{Gee09} T.~Gee, \emph{The Sato-Tate conjecture for modular forms of weight 3}.  Doc. Math.  14  (2009), 771--800.

\bibitem[Gre]{Greene}J.~Greene, \emph{Hypergeometric functions over finite fields}. Trans. Amer. Math. Soc. 301 (1987), no. 1, 77--101.






\bibitem[HLV]{HLV} J.W.~Hoffman, L.~Long and H.~Verrill, \emph{On $\ell$-adic representations for a space of noncongruence cuspforms.} Proc. Amer. Math. Soc. 140 (2012), no. 5, 1569--1584.



\bibitem[KW]{KW09}  C.~Khare and J.-P. Wintenberger,
\emph{On Serre's conjecture for 2-dimensional mod $p$ representations of ${\rm Gal}(\overline{\Bbb Q}/\Bbb Q)$.}
Ann. of Math. (2) 169 (2009), no. 1, 229--253.

\bibitem[Kis]{Ki9} M. Kisin, \emph{Modularity of 2-adic Barsotti-Tate representations}. Invent. Math. 178(3) (2009), 587--634.


\bibitem[Lar]{Lar95} M. ~Larsen, \emph{Maximality of Galois actions for compatible systems}. Duke Math.
J. 80 (1995), 601-630.

\bibitem[Lau]{BL} B.~Laurent,  \emph{An introduction to the theory of $p$-adic representations}. Geometric aspects of Dwork theory. Vol. I, II, 255--292, Walter de Gruyter GmbH \& Co. KG, Berlin, 2004.

\bibitem[LMFDB]{lmfdb} The LMFDB Collaboration, \emph{The L-functions and Modular Forms Database}, \emph{http://www.lmfdb.org} (2013), [Online; accessed 16 September 2013].

\bibitem[LL]{LL} W.-C. W.~Li and L.~ Long, \emph{Fourier coefficients of noncongruence cuspforms}. Bull. Lond. Math. Soc. 44 (2012), no. 3, 591--598.

\bibitem[LLY]{LLY} W.-C. W.~Li, L.~ Long and Z.~Yang, \emph{On Atkin-Swinnerton-Dyer congruence relations.} J. Number Theory 113 (2005), no. 1, 117--148.

\bibitem[LY]{LY16} T.~Liu and J.-K.~ Yu, \emph{On automorphy of certain Galois representations of $GO_4$-type}.
With an appendix by Liang Xiao.
J. Number Theory 161 (2016), 49--74.

\bibitem[Lon]{Long08} L.~Long, \emph{On Atkin and Swinnerton-Dyer congruence relations}. III. J. Number Theory 128 (2008), no. 8, 2413--2429.

\bibitem[Mil]{Mil} J. ~Milne, \emph{Lie Algebras, Algebraic Groups, and Lie Groups}. Note,  available at {\tt{http://www.jmilne.org/math/CourseNotes/LAG.pdf}}

\bibitem[Ram]{Ramakrishnan} D.~Ramakrishnan, \emph{Modularity of the Rankin-Selberg $L$-series, and multiplicity one for ${\rm SL}(2)$}. Ann. of Math. (2) 152 (2000), no. 1, 45--111.




\bibitem[Sch1]{sch85b} A. J. Scholl, \emph{Modular forms and de Rham cohomology; Atkin-Swinnerton-Dyer congruences.}
Invent. Math. 79 (1985), no. 1, 49--77.

\bibitem[Sch2]{sch90} A. J. Scholl, \emph{Motives for modular forms.} Invent. Math. 100 (1990), no. 2, 419--430.

\bibitem[Ser1]{Serre} J-P. ~Serre, \emph{Abelian $l$-adic representations and elliptic curves.} With the collaboration of Willem Kuyk and John Labute. Revised reprint of the 1968 original. Research Notes in Mathematics, 7. A K Peters, Ltd., Wellesley, MA, 1998. 199 pp.

\bibitem[Ser2]{Tate} J-P. ~Serre, \emph{Modular forms of weight one and Galois representations}. Algebraic
number fields: L-functions and Galois properties (Proc. Sympos., Univ. Durham, Durham, 1975), Academic Press, London, 1977, pp. 193--268.

\bibitem[Shi1]{Shimura} G.~Shimura, \emph{Introduction to the arithmetic theory of automorphic functions}. Kan\^o Memorial Lectures, No. 1. Publications of the Mathematical Society of Japan, No. 11. Iwanami Shoten, Publishers, Tokyo; Princeton University Press, Princeton, N.J., 1971.

\bibitem[Shi2]{Shioda} T.~Shioda,  \emph{On elliptic modular surfaces}. J. Math. Soc. Japan 24 (1972), 20--59.

\bibitem[SB]{Beukers-Stienstra} J.~Stienstra and F.~Beukers,
\emph{On the Picard-Fuchs equation and the formal Brauer group of certain elliptic K3-surfaces. }
Math. Ann. 271 (1985), no. 2, 269--304.

\bibitem[Tat]{Tate-62-ICM}
J.~Tate,
\emph{Duality theorems in Galois cohomology over number fields}. 1963 Proc. Internat. Congr. Mathematicians (Stockholm, 1962)  288--295 Inst. Mittag-Leffler, Djursholm.

\end{thebibliography}
\end{document}